\documentclass{amsart}[12]

\usepackage[a4paper]{geometry}
\usepackage{amsmath}
\usepackage{amsthm}
\usepackage{mathrsfs}
\usepackage{amsfonts}
\usepackage{amssymb}
\usepackage{amscd}
\usepackage{array}
\usepackage{amssymb}
\usepackage[all, cmtip]{xy}
\usepackage{tikz}
\usetikzlibrary{arrows}
\usetikzlibrary{cd}
\usepackage{enumitem}

\usepackage{hyperref}

\newtheorem{theorem}{Theorem}[subsection]
\newtheorem{lemma}[theorem]{Lemma}

\newtheorem{definition}[theorem]{Definition}
\newtheorem{corollary}[theorem]{Corollary}
\newtheorem{proposition}[theorem]{Proposition}
\newtheorem{remark}[theorem]{Remark}

\newcommand{\CC}{\mathbb{C}}
\newcommand{\QQ}{\mathbb{Q}}

\newcommand{\PP}{\mathbb{P}}
\newcommand{\ZZ}{\mathbb{Z}}

\newcommand{\EE}{\mathbb{E}}
\newcommand{\LL}{\mathbb{L}}
\renewcommand{\AA}{\mathbb{A}}

\newcommand{\OO}{\mathscr{O}}
\newcommand{\Hom}{Hom}

\newcommand{\n}{\mathfrak{n}}
\newcommand{\m}{\mathfrak{m}}

\newcommand{\Sec}[2]{{\mathrm{Sec}}(#1/#2)}
\newcommand{\Mbar}{{\overline{\mathcal{M}}}}
\newcommand{\C}{\mathcal{C}}
\newcommand{\B}{\mathcal{B}}

\newcommand{\F}{\mathcal{F}}

\newcommand{\E}{\mathscr{E}}

\newcommand{\sslash}{\mathord{/\mkern-6mu/}}

\newcommand{\omegabul}{\omega^\bullet}
\newcommand{\FF}{\mathbb{F}}
\newcommand{\coneE}{\mathfrak{E}}
\newcommand{\coneC}{\mathfrak{C}}

\newcommand{\defY}{\mathcal{Y}}
\newcommand{\defZ}{\mathcal{Z}}

\newcommand{\defE}{E_\defY^-}
\newcommand{\defsE}{\E_\defY^-}

\newcommand{\Mgnt}{{\mathfrak{M}}}
\newcommand{\cM}{\mathcal{M}}
\newcommand{\Ct}{\mathfrak{C}}

\renewcommand{\AA}{\mathbb{A}}

\newcommand{\fP}{\mathfrak{P}}

\newcommand{\vir}{\mathrm{vir}}
\newcommand{\loc}{\mathrm{loc}}

\newcommand{\mv}{\mathrm{mov}}

\DeclareMathOperator{\vb}{vb}
\renewcommand{\H}{\vb(\omega)}

\newcommand{\W}{\mathcal{W}}

\DeclareMathOperator{\fix}{fix}

\DeclareMathOperator{\Spec}{Spec}
\DeclareMathOperator{\Ext}{Ext}
\renewcommand{\Hom}{\mathrm{Hom}}

\DeclareMathOperator{\rk}{rk}
\DeclareMathOperator{\Sym}{Sym}
\DeclareMathOperator{\Pic}{Pic}
\newcommand{\fa}{\mathfrak{a}}
\newcommand{\fU}{\mathfrak{U}}
\newcommand{\D}{\mathrm{D}}
\newcommand{\Dqc}{\mathrm{D}_\mathrm{qc}}

\usepackage{xcolor}

\begin{document}

\title{Virtual cycles of stable (quasi-)maps with fields}
\author{Qile Chen \and Felix Janda \and Rachel Webb}

\address[Q. Chen]{Department of Mathematics\\
Boston College\\
Chestnut Hill, MA 02467\\
U.S.A.}
\email{qile.chen@bc.edu}

\address[F. Janda]{Department of Mathematics \\
  University of Notre Dame \\
  Notre Dame, IN 46556 \\
  U.S.A}
\email{fjanda@nd.edu}

\address[R. Webb]{Department of Mathematics\\
University of California, Berkeley\\
Berkeley, CA 94720\\
U.S.A.}
\email{rwebb@berkeley.edu}

\date{\today}

\begin{abstract}
  We generalize the results of Chang--Li, Kim--Oh and Chang--Li on the
  moduli of $p$-fields to the setting of (quasi-)maps to complete
  intersections in arbitrary smooth Deligne--Mumford stacks with
  projective coarse moduli. 
  In particular, we show that the virtual cycle of stable (quasi-)maps to a complete intersection can be recovered by the cosection localized virtual cycle of the moduli of $p$-fields of the ambient space. 
  \end{abstract}



\maketitle

\setcounter{tocdepth}{1}
\tableofcontents

\section{Introduction}

This paper generalizes Chang--Li's work \cite{CL12} on the moduli of
stable maps with $p$-fields for the quintic hypersurface in $\PP^4$ to
a theory of stable maps with $p$-fields for complete intersections in
smooth Deligne-Mumford stacks (Theorem \ref{thm:mainthm}).

As applications, we reprove a weak version of quantum Lefschetz in
genus 0 (Theorem \ref{thm:lefschetz}), and most importantly, in all
genus, we derive a formula for the virtual cycle of the moduli space
of stable (quasi-)maps to a complete intersection as a virtual cycle
on a moduli space of $p$-fields, which has simpler geometry
(Corollaries \ref{cor:stablemaps} and \ref{cor:quasimaps}).
In combination with the theory of the logarithmic Gauged Linear Sigma
Model (GLSM) in \cite{CJR19P1}, this formula provides a tool for
studying the higher genus Gromov--Witten invariants of complete
intersections.
This formula is new in the case when the target is not a GIT quotient.
Our approach to the $p$-fields problem differs from those in the
literature as we use a setup reminiscent of a GLSM.

To understand how the moduli of $p$-fields is a tool for studying
higher-genus Gromov--Witten invariants, recall the strategy of the
original proof in {\cite{Gi98b, LLY97}} of the genus-zero mirror
theorem for the quintic hypersurface $X \subset \PP^4$.
The first step uses the \textit{quantum Lefschetz principle}
\cite{KKP03, CoGi07} to relate the Gromov-Witten invariants of $X$ to
those of $\PP^4$.
Indeed, a weak version of the quantum Lefschetz principle states
\begin{equation}\label{eq:intro1}
\iota_*[\overline{\cM}_{0,n}(X, \beta)]^{\vir} = e(R^0\pi_*f^*\OO(5))\cap [\overline{\cM}_{g,n}(\PP^4, \iota_*\beta)]^\vir
\end{equation}
(see Theorem \ref{thm:lefschetz}).
The second step uses torus localization on
$\overline{\cM}_{g, n}(\PP^4, \iota_* \beta)$.
(The quintic $X$ itself does not carry an adequate torus action.)

The naive version of the quantum Lefschetz principle fails in positive
genus.
However, an equality similar to \eqref{eq:intro1} holds when we
replace $\overline{\cM}_{g,n}(\PP^4, \iota_*\beta)$ with the
\textit{moduli of stable maps with $p$-fields}
{$\cM_{g,n}(\PP^4, \OO(5), \iota_*\beta)$}.
{The new moduli space} $\cM_{g,n}(\PP^4, \OO(5), \iota_*\beta)$
is a cone over $\overline{\cM}_{g,n}(\PP^4, \iota_*\beta)$ and in
particular it carries a nontrivial torus action.
{On the other hand, the main result in \cite{CL12} can be lifted to an equality of virtual cycles}
\begin{equation}\label{eq:intro2}
[\Mbar_{g,n}(X, \beta)]^\vir = \pm[\cM_{g,n}(\PP^4, \OO(5), \iota_*\beta)]^\vir_\loc
\end{equation}
(see also Corollary \ref{cor:stablemaps}), where the right hand side
is the \textit{cosection localized virtual class} of \cite{KiLi13}.
%
{Equality} \eqref{eq:intro2} is
leveraged to compute higher genus Gromov--Witten invariants of the
quintic { using mixed-spin-$p$-fields by Chang--Guo--Li--Li--Liu
  in \cite{CLLL15P, CLLL16P, CGLL18P, CGL18Pa, CGL18Pb}, and using log
  GLSM by Guo, Ruan, and the first and second author in \cite{CJR19P1,
    CJR19P2, CJR19P3, CJR19P3, GJR17P, GJR18P}.}

We remark that there are several other approaches to quantum Lefschetz
and the computation of the Gromov-Witten invariants of the quintic in
genus one, including \cite{VZ08, Zi08, KiLh18}.

{Given the success of the $p$-fields aproach, it is natural to
  ask if \eqref{eq:intro2} holds for} an arbitrary {smooth}
complete intersections $X$ in smooth Deligne--Mumford stacks $Y$ with
projective coarse moduli.
The main result of this paper (Theorem \ref{thm:mainthm})
{gives a complete answer to this question,} which
{implies that (see Corollary \ref{cor:stablemaps})}
\[
[\cM(X)]^{\vir} = \pm[\cM(Y, E)]^{\vir}_{\loc}
\]
when $\cM(X)$ is a moduli stack of stable (quasi-)maps to $X$,
$\cM(Y, E)$ is a moduli stack of $p$-fields defined by the vector
bundle $E$ whose section cuts out $X$, and the sign is a function of
various data defining the moduli. When $Y$ is a GIT quotient of an
affine variety, our theorem recovers the main results of
\cite{KiOh18P, CL18}. Since the first version of this article appeared
on the arXiv, R. Picciotto \cite{Pi20P}  found an alternative approach
to Corollary~\ref{cor:stablemaps}.

Our approach unifies these special cases in a single theory. We
anticipate applications in both the GIT and non-GIT settings {where
  the latter fits the general set-up of log GLSM.}
As an example of the latter, we hope that Theorem \ref{thm:mainthm},
combined with the log GLSM, can be used to investigate the conjectures
of Oberdieck--Pixton \cite{ObPi19} in the case of Weierstrass
elliptic fibrations.

\subsection{Main definition and result} 
\label{sec:introdefs}

We now state precise definitions and results.
In order to both address stable maps and quasi-maps, we work in the
following abstract set-up.
Fix an algebraic stack $Y$ whose open Deligne--Mumford locus
$Y^0 \subset Y$ is smooth and dense in $Y$.
Let $E$ be a vector bundle on $Y$ with $\E$ its sheaf of sections, and
$s$ be a global section of $E$.
Let $X\subset Y$ denote the zero locus of $s$.
Let $\Mgnt:={\mathfrak{M}_{g,n}^{\mathrm{tw}}}$ denote the moduli
space of prestable twisted curves with genus $g$ and $n$ markings (see
\cite[Section~4]{AV02}), with $\Ct$ its universal curve. 

By \cite[Thm~1.2]{HaRy19} there is an algebraic stack $\Hom_{\Mgnt}(\Ct, Y\times \Mgnt)$ of morphisms from $\Ct$ to $Y\times \Mgnt$ over $\Mgnt$. Choose an open substack $\cM(Y) \subset \Hom_{\Mgnt}(\Ct, Y\times \Mgnt)$, letting $\cM(X) = \cM(Y) \times_{\Hom_{\Mgnt}(\Ct, Y\times \Mgnt)} \Hom_{\Mgnt}(\Ct, X\times \Mgnt)$ denote the moduli of maps in $\cM(Y)$ that factor through $X \subset Y$, such that the following conditions hold:
\begin{enumerate}
\item The stack $\cM(Y)$ is separated, Deligne--Mumford, and finite type.
\item On both $\cM(X)$ and $\cM(Y)$, the canonical construction
  \eqref{eq:candidatepot} is a perfect obstruction theory.
\item For any closed point $[f\colon C \rightarrow Y] \in \cM(Y)$, the locus $f^{-1}(Y^0)$ is open dense in $C$.
\item The universal curve $\mathcal{C} \to \cM(X)$ admits a relatively ample line bundle.
\end{enumerate}
For example, one may take $\cM(Y)$ to be moduli space of
stable maps (see Section~\ref{sec:stablemaps}) or $\epsilon$-stable
quasimaps (see Section~\ref{sec:quasimaps}).

Motivated by the original definition in \cite{CL12}, we define the moduli stack  $\cM(Y, E)$  of {\em maps in $\cM(Y)$ with $p$-fields} to be a certain cone over $\cM(Y)$ as follows. Let $\omega$ denote the relative dualizing sheaf of $\Ct \rightarrow \Mgnt$. Let $Z$ be the vector bundle on $\Ct \times Y$ whose sheaf of sections is $\E \boxtimes \omega^\vee$. As in \cite[Thm~1.3]{HaRy19} define the moduli of sections $\Sec{Z}{\Ct}$ to be the (algebraic) stack over $\Mgnt$ with fibers
\[
\Sec{Z}{\Ct}(T) = \Hom_{\Ct}(\Ct \times_\Mgnt T, Z) = \Hom_{\Ct \times_\Mgnt T}(\Ct \times_\Mgnt T, Z \times_\Mgnt T).
\]
The projection $Z \rightarrow \Ct \times Y$ induces a morphism $\Sec{Z}{\Ct}\rightarrow\Hom_{\Mgnt}(\Ct, Y\times \Mgnt)$. The stack $\cM(Y, E)$ is the fiber product
\[
\cM(Y, E) = \cM(Y) \times_{\Hom_{\Mgnt}(\Ct, Y\times \Mgnt)} \Sec{Z}{\Ct}.
\]
It carries a perfect obstruction theory \eqref{eq:candidatepot} inherited from $\Sec{Z}{\Ct}$
and, under additional assumptions, a cosection \eqref{eq:defcosec} determined by
$s$ whose degeneracy locus is contained in $\cM(X)$ (Corollary \ref{cor:specialcosec}).

To introduce the main result, we define a locally constant function
\begin{equation}\label{eq:virtual-rank-intro}
\chi(E, \cM(X)) \colon  \cM(X) \to \ZZ
\end{equation}
with $\chi(E, \cM(X))([f]) = \chi(f^*\E)$ for any closed point $[f] \in \cM(X)$.\footnote{The function $\chi(E, \cM(X))$ is called the ``virtual rank'' in \cite{KiOh18P}.} 
Since $\chi(E, \cM(X))$ is locally constant on $\cM(X)$, this defines a class
\[
(-1)^{\chi(E, \cM(X))} \in A^0(\cM(X)).
\]
See Remark \ref{rmk:virtual-rank} below for more details, including an explicit formula for
$\chi(E, \cM(X))([f])$.

\begin{theorem}\label{thm:mainthm}
The stack $\cM(Y, E)$ is a separated, finite-type, Deligne--Mumford stack carrying a canonical perfect obstruction theory \eqref{eq:candidatepot}. Furthermore, if $s$ is a regular section and $X \cap Y^0$ is smooth, then there is a cosection localized virtual cycle $[\cM(Y, E)]^{\vir}_{\loc} \in A_*(\cM(X))$ satisfying
\begin{equation}\label{eq:thma}
i_*[\cM(Y, E)]^{\vir}_{\loc} = [\cM(Y, E)]^{\vir} \in A_*(\cM(Y, E))
\end{equation}
and
\begin{equation}\label{eq:thmb}
[\cM(Y, E)]^{\vir}_{\loc} = (-1)^{\chi(E,\cM(X))}[\cM(X)]^{\vir} \in A_*(\cM(X)).
\end{equation}
\end{theorem}

\begin{remark}
  The equality \eqref{eq:thma} is a direct application of
  \cite[Thm~1.1]{KiLi13}, though we take a simplified approach to defining cosection localized virtual fundamental classes that applies when the cosection is defined everywhere (see Appendix \ref{sec:cosection-localization}).
  The bulk of this paper is devoted to the construction of
  $\cM(Y, E)$, its perfect obstruction theory, and cosection; and to
  proving equation \eqref{eq:thmb}.
\end{remark}

\subsection{Applications}

Theorem \ref{thm:mainthm} applies to both stable maps and
$\epsilon$-stable quasimaps.
For stable maps, take $Y$ to be a smooth projective Deligne--Mumford
stack.
Let $E$ be a vector bundle on $Y$ with a regular section whose zero
locus $X$ is smooth.
For nonnegative integers $g, n$ and a class
$\beta \in H_2(\underline{Y})$, where $\underline{Y}$ is the coarse
moduli space of $Y$, let $\Mbar_{g, n}(Y, \beta)$ denote the moduli
stack of twisted stable maps defined in \cite{AV02}.
We set $\cM(Y) = \Mbar_{g, n}(Y, \beta)$.
Furthermore, let $\Mbar_{g,n}(X, \beta)$ be the substack of maps that
factor through $X$; this is the disjoint union of
$\Mbar_{g,n}(X, \beta')$ for all $\beta' \in H_2(\underline X)$ such
that $\iota_* \beta' = \beta$ where
$\iota\colon \underline X \to \underline Y$ is the inclusion.
Then we get the following corollary (Section \ref{sec:stablemaps}).
\begin{corollary}
  \label{cor:stablemaps}
  We have the following identity of virtual classes
  \begin{equation*}
    [\cM(Y, E)]^\vir_\loc = \sum_{\beta'\colon \iota_* \beta' = \beta} (-1)^{\chi(E, \;\Mbar_{g,n}(X, \beta'))}[\Mbar_{g,n}(X, \beta')]^\vir \quad \quad \text{in} \; A_*(\Mbar_{g,n}(X, \beta)).
  \end{equation*}
\end{corollary}
\begin{remark}
  We have
  \begin{equation*}
    A_*(\Mbar_{g,n}(X, \beta))
    = \bigoplus_{\beta'\colon \iota_* \beta' = \beta} A_*(\Mbar_{g,n}(X, \beta')),
  \end{equation*}
  and hence the identity allows one to recover each
  $[\Mbar_{g,n}(X, \beta')]^\vir$ from $[\cM(Y, E)]^\vir_\loc$.
  In many situations, ordinary Lefschetz implies that $\iota_*$ is an
  isomorphism, so that the direct sum has exactly one term.

\end{remark}

In this stable map setup, we can also derive a weak version of the quantum Lefschetz theorem from Theorem \ref{thm:mainthm}. Recall that $\E$ is \textit{convex} if for every closed point
$[f\colon C \to Y] \in \cM(Y)$, we have
\begin{equation*}
  H^1(C, f^* \E) = 0.
\end{equation*}
Let $f\colon C \to Y$ denote the universal map on the universal
  curve $\pi\colon C \to \cM(Y)$ and let $\iota\colon \cM(X) \rightarrow \cM(Y)$ denote the inclusion. The strongest version of quantum Lefschetz \cite{KKP03} says that when $\E$ is convex, we have
\begin{equation}\label{eq:stronglef}[\cM(X)]^\vir= \iota^![\cM(Y)]^\vir\end{equation}
where $\iota^!$ is the Gysin pullback. Theorem \ref{thm:mainthm} implies the following weaker version of this statement (Section \ref{sec:lefschetz}).
\begin{theorem}
  \label{thm:lefschetz}
  If $\E$ is convex, then
the direct image sheaf $R^0 \pi_* f^* \E$ is locally free, and
  \begin{equation*}
    \iota_* [\cM(X)]^\vir = e(R^0 \pi_* f^* \E) \cap [\cM(Y)]^\vir \qquad \text{in }A_*(\cM(Y)).
  \end{equation*}
\end{theorem}
Theorem \ref{thm:lefschetz} is a consequence of the identity \eqref{eq:stronglef} after applying $\iota_*$ to both sides (note that the section $s$ induces a section $R^0 f^* s$ of $R^0 \pi_* f^* \E$
whose zero locus is $\cM(X) \subset \cM(Y)$).

\bigskip

A second application of Theorem \ref{thm:mainthm} is to take $Y=[W/G]$
where $W$ is an affine l.c.i. variety and $G$ is a reductive group
acting on $W$.
Choose a character $\theta$ of $G$ such that
$W^s_{\theta} = W^{ss}_{\theta}$ is smooth and nonempty and has finite
$G$-stabilizers.
Let $E$ be a $G$-equivariant vector bundle on $W$ with a
$G$-equivariant regular section whose zero locus $U$ has smooth
intersection with $W^s_{\theta}$.
This data defines a smooth Deligne--Mumford stack
$W\sslash_{\theta} G:= [W^s_{\theta}/G]$ carrying a vector bundle
induced by $E$ with a regular section whose zero locus
$U\sslash_{\theta}G := [(W^s_{\theta} \cap U)/G]$ is smooth.
Fix nonnegative integers $g, n$ and a positive rational number
$\epsilon$, and choose a class $\beta \in \Hom(\Pic^G(W), \QQ)$.
Let $\cM(Y) = \Mbar_{g,n}^\epsilon(W\sslash_{\theta}G, \beta)$ be the
moduli stack of $\epsilon$-stable quasimaps defined in \cite{CCK15}.
Let $\Mbar_{g,n}^\epsilon(U\sslash_{\theta}G, \beta)$ be the substack
where the quasi-map factors through $U$.
Then we have the following corollary (Section \ref{sec:quasimaps}).
\begin{corollary}
  \label{cor:quasimaps}
  We have the following identity of virtual classes
  \begin{equation*}
    [\cM(Y, E)]^\vir_\loc = \sum_{\substack{\beta' \in \Hom(\Pic^G(U), \QQ) \\ \iota_*\beta' = \beta}} (-1)^{\chi(E,\; \Mbar_{g,n}^\epsilon(U\sslash_{\theta}G, \beta))}[\Mbar_{g,n}^\epsilon(U\sslash_{\theta}G, \beta)]^\vir
  \end{equation*}
  in $A_*(\Mbar_{g,n}^\epsilon(U\sslash_{\theta}G, \beta))$, where
  $\iota_*$ is the dual of the map
  $\iota^*\colon \Pic^G(W) \to \Pic^G(U)$ induced by the inclusion
  $\iota\colon U \to W$.
\end{corollary}

\subsection{Contents of the paper}
Our proof of Theorem \ref{thm:mainthm} follows roughly the strategy of \cite{CL12}. We construct (Section \ref{sec:target}) an auxiliary moduli space $\cM \rightarrow \Mgnt\times \AA^1$. This space is roughly analogous to $\cM(Y, E)$ but with $Y$ replaced by the deformation to the normal cone of $X$ in $Y$, and in fact a generic fiber of $\cM$ is isomorphic to $\cM(Y, E)$. We show that $\cM$ (resp.\ its fibers) is a Deligne--Mumford stack with the necessary properties (Section \ref{sec:target}) carrying a canonical perfect obstruction theory relative to $\Mgnt \times \AA^1$ (resp.\ $\Mgnt$; Section \ref{sec:pot}). 

The section $s$ induces a cosection $\sigma$ of the perfect obstruction theory on $\cM$ with degeneracy locus $\cM(X)\times \AA^1$; moreover $\sigma$ specializes to cosections on the fibers of $\cM$ with degeneracy loci $\cM(X)$ (Section \ref{sec:cosec}). A torus localization argument shows that the cosection localized virtual cycle on the special fiber is equal to the usual virtual cycle on $\cM(X)$, up to a sign (Section \ref{sec:specialfiber}), while the cosection localized class of the generic fiber is precisely the class $[\cM(Y, E)]^{\vir}_{\loc}$ of Theorem \ref{thm:mainthm}. 

Section \ref{sec:applications} elaborates on the main applications of our result, Appendix~\ref{sec:moduli-sections} contains several technical lemmas related to the moduli of sections, and Appendix~\ref{sec:cosection-localization} discusses a simplified construction of cosection-localized virtual cycles.

\subsection{Conventions}\label{sec:conventions}
By a sheaf on an algebraic stack $Y$ we mean a sheaf in the lisse-\'etale topos $Y_{lis-et}$, unless otherwise stated. By the derived category $D(Y)$ we mean the unbounded derived category of $\OO_Y$-modules in $Y_{lis-et}$ with quasi-coherent cohomology, as defined in \cite{HR17}. Our derived functors are the ones defined in (loc. cit.). 

If $\mathscr{P}$ is a principal $\CC^*$-bundle on an algebraic stack $X$ and $V$ is a $G$-stack, then $\mathscr{P}\times_{\CC^*} V$ is the quotient of $\mathscr{P} \times V$ by $\CC^*$, where $\CC^*$ acts with the diagonal action.

The universal objects on the moduli of sections $\Sec{Z}{\Ct}$ (or an open substack $\cM$) will usually be denoted $\pi_{\cM}\colon \C_{\cM} \rightarrow \cM$ for the universal curve and $\n_{\cM}\colon \C_{\cM} \rightarrow Z$ for the universal section. We will use $\omega_{\cM}$ to denote the relative dualizing sheaf of the morphism $\C \rightarrow \cM$. We will use $\omegabul_{\cM}=\omega_{\cM}[1]$ to denote the dualizing object in the derived category. If the subscript $\cM$ on any of these notations can be safely deduced from context we may omit it. One consistent exception to our convention is when $Z = \C \times Y$ for some algebraic stack $Y$. In this case $\Sec{\C \times Y}{\C}$ is canonically identified with the moduli of prestable maps to $Y$, and we use $f\colon \C_{\Sec{\C \times Y}{\C}} \rightarrow \C\times Y$ for the universal section.

\subsection{Acknowledgments}

The first author was partially supported by NSF grant DMS-1700682 and DMS-2001089. 
The second author was partially supported by an AMS--Simons travel grant and
NSF grants DMS-1901748 and DMS-1638352.
The third author was partially supported by an NSF Postdoctoral Research Fellowship, award number DMS-2002131. The third author would like to thank Tom Graber, Daniel Halpern-Leistner, Melissa Liu, and Martin Olsson for helpful discussions.
The authors thank Jonathan Wise, who shared with us several deformation arguments that greatly simplified the proof; Jack Hall, who provided a crucial lemma for the functoriality needed in Lemma \ref{lem:technical}; Bumsig Kim and Jeongseok Oh, who explained the proof of Lemma \ref{lem:global-res}; and Richard Thomas for suggestions for improving the introduction.
The authors are grateful to the AGNES conference and the Casa Matem\'atica Oaxaca which facilitated the completion of this project.

\section{Deformation to the normal cone}\label{sec:target}

\subsection{The family of targets}
Recall the notations defined in Section \ref{sec:introdefs}; in particular, $Y$ is an algebraic stack with locally free sheaf $\E$ and a global section $s$. We recall the definition of regularity for $s$.

\begin{definition}\label{def:regular-section}
The section $s$ is \textit{regular} if smooth affine locally, components of $s$ form a regular sequence.
\end{definition}

This definition makes sense when $Y$ is an algebraic stack because regularity is preserved by flat base change (see \cite[067P]{stacks-project}).

From now on we will assume that $s$ is a regular section and that $X \cap Y^0$ is smooth. Some parts of our construction do not require these assumptions, and we will indicate where this is the case.

Let $I_X$ be the ideal sheaf of $X$ in $Y$, and $J = (I_X, t)$ be the ideal sheaf of $X\times 0$ in $Y\times\AA^1$ where $t$ is the coordinate of $\AA^1$. Consider $\overline{\defY} \to Y\times \AA^1$ the blow-up of $Y\times\AA^1$ along the ideal $J$. Let $\defY$ be obtained by removing the proper transform of $Y\times 0$ from $\overline{\defY}$. Then we have a flat family of embeddings
\begin{equation}\label{eq:deformation}
\xymatrix{
X\times\AA^1 \ar[rr]^{\iota} \ar[rd]_{p_2} & & \defY \ar[ld]^{\rho_{\AA^1}}\\
& \AA^1 &
}
\end{equation}
Indeed, the family $\rho_{\AA^1}$ is the deformation to the normal cone of $X$ in $Y$ \cite[Section~5.1]{Fulton}. For any $c \in \AA^1$, denote by $\defY_c$ the fiber of $\rho_{\AA^1}$ over $c$. Then we have  
\begin{equation}\label{eq:target-fiber}
\defY_c \cong Y \ \ \mbox{for any } c \neq 0, \ \ \  \mbox{and} \ \ \  \defY_0 \cong N_{X/Y},
\end{equation} 
where $N_{X/Y}$ is the normal bundle to $X$ in $Y$. Note that $N_{X/Y}$ is a vector bundle over $X$ as the embedding $X \hookrightarrow Y$ is lci.

Note that the morphism $s^{\vee}\colon  \E^{\vee} \to \OO_Y$ factors through $I_X \subset \OO_Y$, hence induces a surjection 
\[
\E^{\vee}|_{Y\times \AA^1} \oplus \OO_{Y\times \AA^1} \to J.
\]
This defines a surjection of graded algebras
\[
\Sym^{\bullet}_{Y\times \AA^1} (\E^{\vee}|_{Y\times \AA^1} \oplus \OO_{Y\times \AA^1}) \to \Sym^\bullet_{Y\times\AA^1} J,
\]
hence a closed embedding
\begin{equation}\label{eq:target-embedding}
\PP(\E|_{Y\times \AA^1} \oplus \OO_{Y\times \AA^1}) \hookleftarrow \overline{\defY}.
\end{equation}
A local calculation shows that this embedding is regular
(see \cite[Lemma~2.1]{Aluffi}).
Furthermore, this embedding restricts to a regular embedding
$ \defY \to E|_{Y\times \AA^1} $.

\subsection{The superpotential}\label{sec:superpotential}

The pullback $s|_{\defY}$ vanishes along the fiber $\defY_0$, hence defines a section $s_\defY^-$ of $\defsE := \E_\defY(-\defY_0)$. The following lemma will be useful for computations later. 

\begin{lemma}\label{lem:computecosec}
The section $s_\defY^-$ of $\defsE$ is a regular section with zero locus $X \times \AA^1 \subset \defY$. Its restriction $s_\defY^-|_{\defY_0}$ is canonically identified with the tautological section of $\E|_{N_{X/Y}} \simeq \E|_{E|_X}$.
\end{lemma}
\begin{proof}
The deformation $\defY$ is constructed smooth locally, so we may assume that $Y = \Spec(A)$ is an affine scheme and $\E$ splits. Let $s \in \Gamma(Y, \E)$ be given by elements $(a_1, \ldots, a_r) \in A$; by Definition \ref{def:regular-section} this sequence is regular. By \cite[Sec~5.1]{Fulton}, $\defY$ is equal to $\Spec(S^\bullet_{(T)}),$ where $S^\bullet_{(T)}$ is the graded ring
\[
S^\bullet_{(T)} = \ldots \oplus I^2T^{-2} \oplus IT^{-1}\oplus A \oplus AT \oplus AT^2 \oplus \ldots.
\]
The ideal of $\defY_0$ is generated by $T$ in degree 1. So while $s|_\defY$ is given by the sequence $(a_1, \cdots, a_r)$, as a section of $\defsE$ it is given by $(a_1T^{-1},\cdots,a_rT^{-1})$. It is straightforward to check that this sequence is regular, and that it generates the ideal of $X \times \AA^1 \subset \defY$. 

Restricting to the fiber $T=0$, we get the coordinate ring
\[
A/I \oplus (I/I^2)T^{-1} \oplus (I^2/I^3)T^{-2}\oplus \ldots
\]
and our section is the sequence $(a_1T^{-1},\cdots a_rT^{-1})$ in degree 1 ($T$ has degree $-1$). Since the sequence is regular it identifies this ring with $Sym^\bullet (A^{\oplus r})$, and under this identification the sequence becomes the tautological one: $(e_1, \ldots, e_r)$ where $e_i$ is 1 in the $i^{th}$ coordinate and 0 elsewhere.
\end{proof}

Let $\defE$ be the vector bundle on $\defY$ whose sheaf of sections is $\defsE$. The dual of $s_\defY^-$ induces a morphism $W$ called the {\em super-potential}:
\begin{equation}\label{eq:family-super-potential}
W\colon (\defE)^{\vee} \to \CC.
\end{equation}
It is linear on the fibers of the vector bundle $(\defE)^\vee =  (E|_{\defY}(\defY_0))^\vee$---in other
words, letting $\CC^*$ act by scaling on both the source and target of
\eqref{eq:family-super-potential}, $W$ is equivariant. The next lemma says that we can study either the zeros of $s_{\defY}^-$ or the critical locus of $W$.

\begin{lemma}\label{lem:critical-loci}
 Suppose $Y$ is a smooth Deligne--Mumford stack and $E$ is a vector bundle on $Y$ with a regular section $s$ and zero locus $X \subset Y$. Let $W\colon E^\vee \rightarrow \CC$ be the function induced by the dual of $s$. Then $X$ is smooth if and only if the critical locus of $W$ equals $X$. 
\end{lemma}
\begin{proof}
  Both sides can be checked \'etale locally, and we may therefore
  assume that $Y$ is a smooth scheme, and that $E$ is trivial of
  rank $r$ on $Y$.
  If we write $s = (s_1, \dotsc, s_r)\colon Y \to \CC^r$,
  then the function $W:Y \times \CC^r \to \CC$ induced by $s$ is given
  by
  \begin{equation*}
    (y, p_1, \dotsc, p_r) \mapsto \sum_{i = 1}^r p_i \cdot s_i(y),
  \end{equation*}
  and has differential
  \begin{equation*}
    \sum_{i = 1}^r [dp_i \cdot s_i(y) + p_i d(s_i(y))].
  \end{equation*}

  Let $Z$ be the critical locus of $W$, which is given by the
  equations $s_i(y) = 0$, $p_i d(s_i(y)) = 0$ for all $i$.
  On the other hand, $X  \subset E^\vee$ is given by
  $s_i(y) = 0$, $p_i = 0$ for all $i$.
  Hence with no assumptions, we have $Z \subset X$.
  
  The other inclusion holds if and only if for every closed
  $y \in X$ the collection of vectors $\{d(s_i(y))\}_i$ are linearly independent. This amounts to saying that the Jacobian matrix of $(s_1, \dotsc, s_r)$ has rank $r$. In particular, $X$ is smooth of codimension $r$. 
\end{proof}

We apply this lemma as follows.
Let $\defY^0 \subset \defY$ denote the
 deformation to the normal cone of $X \cap Y^0$ in $Y^0$; observe that $\defY^0$ is smooth and Deligne--Mumford.
Let $E^0 = \defE \times_\defY \defY^0$ be the restriction
of $\defE$ to $\defY^0$ and let $W^0$ be the corresponding
restriction of \eqref{eq:family-super-potential}.
Let $Crit(W^0) \subset (E^0)^\vee$ be the critical locus of $W^0$.
By Lemmas \ref{lem:critical-loci} and \ref{lem:computecosec}, we have
\begin{equation}\label{eq:critical-locus}
Crit(W^0) = (X \cap Y^0)\times \AA^1
\end{equation} 
Let $\overline{Crit(W^0)}$ be the closure in $(E^0)^\vee$. Hence $\overline{Crit(W^0)} \subset X \times \AA^1$.

\subsection{The moduli}\label{ss:moduli}
Recall from Section \ref{sec:introdefs} that we defined the moduli of $p$-fields $\cM(Y,E)$ as a substack of the \textit{moduli of sections}. We found it convenient to work with these moduli throughout our paper, in particular for constructing perfect obstruction theories. The general construction is as follows.  
Consider a tower of algebraic stacks over $\CC$
\begin{equation}\label{eq:generalsetup}
Z \rightarrow \C\xrightarrow{\pi} \mathfrak{U}
\end{equation}
where $\pi\colon\C\rightarrow \mathfrak{U}$ is a flat finitely-presented family of connected, nodal, twisted curves and $Z\rightarrow \mathfrak{U}$ is locally finitely presented, quasi-separated, and has affine stabilizers. These technical conditions are imposed to guarantee the algebracity of hom-stacks below. We define the \textit{moduli of sections} $\Sec{Z}{\C}$ to be the stack whose fiber over $T \rightarrow \mathfrak{U}$ is \begin{equation}\label{eq:defsec}
\Sec{Z}{\C}(T) = \Hom_{{\C}}(\C\times_\mathfrak{U} T, Z) = \Hom_{{\C\times_\mathfrak{U} T}}(\C\times_\mathfrak{U} T, Z\times_\mathfrak{U} T).\end{equation}
By \cite[Thm~1.3]{HaRy19}, the stack $\Sec{Z}{\C}$ is algebraic and the canonical morphism $\Sec{Z}{\C}\rightarrow \mathfrak{U}$ is locally finitely presented, quasi-separated, and has affine stabilizers.

We first apply this construction to define a deformation of the space $\cM(Y, E)$. Recall the universal family of twisted curves $\Ct \to \Mgnt$ from Section~\ref{sec:introdefs}. Consider $\Mgnt_{\AA^1} = \Mgnt \times \AA^1$ with the universal curve $\Ct_{\AA^1} = \Ct\times \AA^1 \to \Mgnt_{\AA^1}$. Let $\omega$ be the relative dualizing sheaf of $\Ct_{\AA^1} \to \Mgnt_{\AA^1}$. On the algebraic stack $\defY$ denote by $\defsE$ the locally free sheaf $(\rho_Y^*\E)(-\defY_0)$. Define
\begin{equation}\label{eq:universal-target}
\defZ = Vb_{\Ct_{\AA^1} \times_{\AA^1} \defY}(\omega\otimes(\defsE)^\vee).
\end{equation}
The projections $\defZ \rightarrow \Ct_{\AA^1}\times_{\AA^1} \defY \rightarrow \Ct_{\AA^1}\times Y$ induce a morphism 
\begin{equation}\label{eq:proj-to-Y}
\Sec{\defZ}{\Ct_{\AA^1}} \rightarrow \Sec{\Ct_{\AA^1}\times Y}{\Ct_{\AA^1}} \cong \Sec{\Ct\times Y}{\Ct} \times \AA^1.
\end{equation}

\begin{definition}
The deformation moduli space of $p$-fields, denoted $\cM(\defY, \defsE)$, is the fiber product
\[
\cM(\defY, \defsE) = (\cM(Y)\times \AA^1) \times_{\Sec{\Ct_{\AA^1}\times Y}{\Ct_{\AA^1}}} \Sec{\defZ}{\Ct_{\AA^1}}.
\]
\end{definition}
Observe that $\cM(\defY, \defsE)$ is a stack over $\Mgnt_{\AA^1}$, and in particular has a canonical projection to $\AA^1$.
\begin{proposition}\label{prop:goodstack}
The stack $\cM(\defY, \defsE)$ is a separated Deligne--Mumford stack of finite type.
\end{proposition}
\begin{proof}
We have a sequence of morphisms
\[
\Sec{\defZ}{\Ct_{\AA^1}} \rightarrow \Sec{\Ct_{\AA^1}\times_{\AA^1}\defY}{\Ct_{\AA^1}} \rightarrow \Sec{\Ct_{\AA^1}\times \rho_Y^*E}{\Ct_{\AA^1}}\rightarrow \Sec{\Ct_{\AA^1}\times Y}{\Ct_{\AA^1}}
\]
Applying \cite[Prop~2.2]{CL12} for orbifold curves, the first and last are representable by affine schemes of finite type, hence in particular they are separated, finite type, and representable. The middle arrow is a closed embedding by Lemma \ref{lem:closed} since \eqref{eq:target-embedding} is so. Pulling back the above sequence along $\cM(Y)\times \AA^1 \to \Sec{\Ct_{\AA^1}\times Y}{\Ct_{\AA^1}}$, the assumption that $\cM(Y)$ is a separated Deligne--Mumford stack of finite type implies that $\cM(\defY, \defsE)$ has the same properties.
\end{proof}

\subsection{Specialization} \label{sec:special1}
Let $c \in \AA^1$. We compute the fiber $\cM_c$ of $\cM := \cM(\defY, \defsE)$ over $c$, first expressing it in terms of $\defZ_c$. The fiber diagram
\[\begin{tikzcd}
\defZ_c \arrow[r] \arrow[d] & \Ct \times\defY_c \arrow[r]\arrow[d] &
\Ct \times Y \arrow[r] \arrow[d] &
\Ct \arrow[r] \arrow[d] & \Mgnt\times\{c\} \arrow[d] \\
 \defZ \arrow[r]&
 \Ct_{\AA^1}\times_{\AA^1}\defY \arrow[r]&
 \Ct_{\AA^1}\times Y \arrow[r] &
 \Ct_{\AA^1} \arrow[r]&
\Mgnt\times \AA^1 
\end{tikzcd}
\]
induces the following fiber diagram with isomorphic horizontal arrows:
\[
\begin{tikzcd}
\Sec{\defZ}{\Ct_{\AA^1}}\times_{\Mgnt\times \AA^1} (\Mgnt\times\{c\}) \arrow[r] \arrow[d] & \Sec{\defZ_c}{\Ct} \arrow[d] \\
\Sec{\Ct_{\AA^1}\times Y}{\Ct_{\AA^1}}\times_{\Mgnt\times \AA^1} (\Mgnt\times\{c\}) \arrow[r] & \Sec{\Ct \times Y}{\Ct} 
\end{tikzcd}
\]
Pulling back this square over $\cM(Y) \subset \Sec{\Ct\times Y}{\Ct}$, we see that
\begin{equation}\label{eq:specialm}
\cM_c = \cM(Y) \times_{\Sec{\Ct\times Y}{\Ct}} \Sec{\defZ_c}{\Ct}.
\end{equation}
When $c \neq 0$, by \eqref{eq:target-fiber} we have $\cM_c = \cM(Y, E)$ the moduli constructed in Section~\ref{sec:introdefs}.

To simplify \eqref{eq:specialm} when $c=0$, we first compute the restriction:
\[
  \defsE|_{\defY_0} = \left( \E_{\defY}\right)|_{\defY_0}\otimes \left( \OO(-\defY_0)\right)|_{\defY_0} = \E|_{\defY_0} \otimes \left( \OO_{\AA^1}(-0)\right)|_{\defY_0} \cong \E|_{\defY_0},
\]
since $\defY_0$ is the fiber over the origin $0 \in \AA^1$.
We conclude that the special fiber $\cM_0$ is equal to
\begin{equation}\label{eq:specialm0}
\cM_0 = \cM(Y) \times_{\Sec{\Ct \times Y}{\Ct}} \Sec{Vb_{\Ct \times \defY_0}(\omega \otimes \E^\vee|_{\defY_0})}{\Ct} = \cM(X) \times_{\Sec{\Ct \times X}{\Ct}} \Sec{\defZ_0}{\Ct}
\end{equation}
where $\defZ_0 = Vb_{\Ct \times X}(\E_X \oplus \omega \otimes \E_X^\vee)$. The final equality uses that $\defY_0 = N_{X/Y} \cong \E|_X$ since differentiating the regular section $s$ induces a surjection $T_Y|_X \to \E|_X$ with kernel $T_X$.
In particular we have the following corollary to Proposition \ref{prop:goodstack}.
\begin{corollary}
  The stack $\cM(Y, E)$ is a separated Deligne--Mumford stack of finite type.
\end{corollary}
\begin{remark}
The proof of Proposition \ref{prop:goodstack} can also be used to directly show that $\cM(Y, E)$ is a separated Deligne--Mumford stack of finite type. In particular, no properties of $s$ are needed for this result.
\end{remark}

\section{The perfect obstruction theory}\label{sec:pot}

\subsection{A general set-up}
We now construct a ``candidate'' perfect obstruction theory for the moduli of sections defined in \eqref{eq:defsec}, in the general situation of \eqref{eq:generalsetup}. Consider the family $Z \rightarrow \C \rightarrow \mathfrak{U}$ in \eqref{eq:generalsetup}, and recall the notational conventions in \ref{sec:conventions}.

Let $\LL_{Z/\C}$ denote the relative cotangent complex of \cite{olsson07}. We have a morphism in the derived category of $\C_{Sec}$ 
\[
L\n^*\LL_{Z/\C} \rightarrow \LL_{\C_{Sec}/\C} \xleftarrow{\sim} \pi^*\LL_{\Sec{Z}{\C}/\mathfrak{U}}.
\]
Tensoring this morphism with $\omegabul_{Sec}$ and applying $R\pi_*$, we obtain
\[
R\pi_*(L\n^*\LL_{Z/\C} \otimes \omegabul_{Sec}) \rightarrow R\pi_*(\pi^*\LL_{\Sec{Z}{\C}/\mathfrak{U}} \otimes \omegabul_{Sec}).
\]
Since $\pi^!\bullet = \omegabul_{Sec} \otimes \pi^*\bullet$ is right adjoint to $R\pi_*$, we obtain
\[
R\pi_*(\pi^*\LL_{\Sec{Z}{\C}/\mathfrak{U}} \otimes \omegabul_{Sec}) = R\pi_*\pi^!\LL_{\Sec{Z}{\C}/\mathfrak{U}} \to \LL_{\Sec{Z}{\C}/\mathfrak{U}} 
\]
hence a morphism
\begin{equation}\label{eq:candidatepot}
\phi_{\Sec{Z}{\C}/\mathfrak{U}}\colon \EE^\bullet_{\Sec{Z}{\C}/\mathfrak{U}}:= R\pi_*(L\n^*\LL_{Z/\C} \otimes \omegabul_{Sec}) \rightarrow \LL_{\Sec{Z}{\C}/\mathfrak{U}}.
\end{equation}

Applying the construction of \eqref{eq:candidatepot} to $Z = \Ct\times Y \to \Ct \to \Mgnt$, we obtain 
\begin{equation}\label{eq:initialpotY}
\phi_{Y} \colon \EE_{\cM(Y)/\Mgnt} := R\pi_*(L f^*\LL_{Y}\otimes\omegabul) \longrightarrow  \LL_{\cM(Y)/\Mgnt}
\end{equation}
where $f\colon \C \to Y$ is the composition of the universal section with the projection $\Ct \times Y \rightarrow Y$. By assumption, $\phi_{Y}$ is a perfect obstruction theory for $\cM(Y) \to \Mgnt$. Similarly replacing $Y$ by $X$, we obtain 
\begin{equation}\label{eq:initialpotX}
\phi_{X} \colon \EE_{\cM(X)/\Mgnt} := R\pi_*(L f^*\LL_{X}\otimes\omegabul) \longrightarrow  \LL_{\cM(X)/\Mgnt}
\end{equation}
which we assume is a perfect obstruction theory of $\cM(X) \to \Mgnt$.

\subsection{The family of perfect obstruction theories}
We investigate some situations when \eqref{eq:candidatepot} is a (perfect) obstruction theory in the sense of \cite{BF97}. Consider the following commutative diagram

\begin{equation}\label{eq:allmoduli}
\xymatrix{
&& && \defZ \ar[d]^q & \\
&& f^{-1}\defZ \ar[d] \ar[urr] && \Ct_{\AA^1}\times Y \ar[d] &\\
\C_{\cM(\defY, \defsE)} \ar[urr]^{\tilde\n} \ar@/^2pc/[uurrrr]^{\n} \ar[rr]^{\mu} \ar[d]^{\pi} && \C_{\cM(Y)_{\AA^1}} \ar[rr] \ar[d] \ar[urr]^{\m} && \Ct_{\AA^1} \ar[d] & \\
\cM(\defY, \defsE) \ar[rr]^{p} &&\cM(Y)_{\AA^1} \ar[rr] && \Mgnt_{\AA^1} \ar[r] & \AA^1\\
}
\end{equation}
where the two bottom squares are cartesian, $\m$ is the universal map of $\cM(Y)$, and $\n$ is the universal map of $\cM(\defY, \defsE)$.  Applying \eqref{eq:candidatepot} to $\cM := \cM(\defY, \defsE) \to \Mgnt_{\AA^1}$, we obtain 
\begin{equation}\label{eq:family-pot}
\phi \colon \EE_{\cM/\Mgnt_{\AA^1}} := R\pi_*(L\n^*\LL_{\defZ/\Ct_{\AA^1}} \otimes \omegabul_{\cM(\defY, \defsE)})  \longrightarrow  \LL_{\cM(\defY, \defsE)/\Mgnt_{\AA^1}}.
\end{equation}

For later use, we describe the fibers of $\phi$ over $\AA^1$ as follows.
Let $c$ be a closed point in $\AA^1$ and let $\iota_c\colon \cM_c \rightarrow \cM(\defY, \defsE)$ be the fiber of $\cM(\defY, \defsE) \to \AA^1$ over $c$ as in Section \ref{sec:special1}. The \textit{specialization} $\phi_c$ of $\phi$ to $\cM_c$ is the composition
\[
\phi_c\colon \EE_{\cM_c/\Mgnt} := \iota_c^*\EE_{\AA^1} \xrightarrow{\iota_c^*\phi} \iota_c^*\LL_{\cM(\defY, \defsE)/\Mgnt_{\AA^1}} \longrightarrow \LL_{\cM_c/\Mgnt}
\]
where the second arrow is the canonical one. By Lemma \ref{lem:functoriality2}, since $\defY \rightarrow \AA^1$ is flat, the morphism $\phi_c$ is isomorphic to the one constructed by applying \eqref{eq:candidatepot} to the fiber of \eqref{eq:allmoduli} over $c \in \AA^1$. That is, we have a commuting diagram
\begin{equation}\label{eq:specialpot}
\begin{tikzcd}
\EE_{\cM_c/\Mgnt} \arrow[d, "\sim"] \arrow[r, "\phi_c"] & \LL_{\cM_c/\Mgnt} \arrow[d, equal]\\
 R\pi_*(L\n_c^*\LL_{\defZ_c/\Ct} \otimes \omegabul_{\cM_c})  \arrow[r]  &\LL_{\cM_c/\Mgnt}
\end{tikzcd}
\end{equation}
We will also refer to the bottom arrow as $\phi_c$.

The main result of this section is the following:

\begin{proposition}\label{prop:defspace} The morphism $\phi$ (resp. $\phi_c$) defines a perfect obstruction theory for $\cM(\defY, \defsE) \to \Mgnt_{\AA^1}$ (resp. $\cM_c \to \Mgnt$).
\end{proposition}

\subsection{Proof of Proposition \ref{prop:defspace}}

By \cite[Theorem~7.2]{BF97}, it suffices to show that $\phi$ is a perfect obstruction theory. We first show that $\phi$ defines an obstruction theory of $\cM(\defY, \defsE) \to \Mgnt_{\AA^1}$ in the sense of \cite[Definition 4.4]{BF97}. 

The composition $\defZ \to \Ct_{\AA^1}\times Y \to \Ct_{\AA^1}$ as in \eqref{eq:allmoduli} induces a triangle of cotangent complexes
\[
Lq^*\LL_{\Ct_{\AA^1}\times Y/\Ct_{\AA^1}}|_{\defZ} \longrightarrow \LL_{\defZ/\Ct_{\AA^1}}|_{\defZ} \longrightarrow \LL_{\defZ/\Ct_{\AA^1}\times Y} \stackrel{[1]}{\longrightarrow}.
\]
By Lemma \ref{lem:functoriality1} we have a morphism of distinguished triangles
\begin{equation}\label{eq:proofpot}
\begin{tikzcd}
 p^*\EE_{\cM(Y)_{\AA^1}/\Mgnt_{\AA^1}}\arrow[r]\arrow[d]& \EE_{\cM/\Mgnt_{\AA^1}} \arrow[r]\arrow[d, "\phi"]&R\pi_*(\n^*\LL_{\defZ/(\Ct_{\AA^1} \times Y)}\otimes \omegabul) \arrow[r]\arrow[d, "\phi_{rel}"]&{}\\
 p^*\LL_{\cM(Y)_{\AA^1}/\Mgnt_{\AA^1}}\arrow[r]  & \LL_{\cM(\defY, \defsE)/\Mgnt_{\AA^1}} \arrow[r] &\LL_{\cM(\defY, \defsE)/\cM(Y)_{\AA^1}}  \arrow[r] & {}
\end{tikzcd}
\end{equation}
We claim that the left and right vertical maps have the property that
$h^0$ is an isomorphism and $h^{-1}$ is surjective.
Granting this, applying the five lemma to the long exact sequence of
cohomology shows that the middle vertical map does as well. 
The claim on the leftmost arrow holds since we have assumed that \eqref{eq:candidatepot} defines an obstruction theory on $\cM(Y)$ over $\Mgnt$, so by Lemma \ref{lem:basechange2} the morphism \eqref{eq:candidatepot} also defines an obstruction theory on $\cM(Y)_{\AA^1}$ over $\Mgnt_{\AA^1}$ (in this case, both horizontal arrows in \eqref{eq:functoriality2} are quasi-isomorphisms). The claim on the rightmost arrow is the content of the following lemma.

\begin{lemma}\label{lem:proofopt}
The morphism $\phi_{rel}$ is an obstruction theory in the sense of \cite[Definition 4.4]{BF97}.
\end{lemma}
The argument for Lemma \ref{lem:proofopt} is standard. For example, the proof of \cite[Prop~4.2]{ACGS-puncture} applies verbatim to our situation, once we have the following lemma:

\begin{lemma}
Suppose we have a commutative diagram of solid arrows
\begin{equation}\label{eq:proofobs}
\xymatrix{
T \ar[d]_{J^2 = 0} && \C_T \ar[d] \ar[ll] \ar[rr]^{\n_T} && \defZ \ar[d]_p \\
T' && \C_{T'} \ar[ll] \ar[rr] \ar@{-->}[rru]^{\n_{T'}} \ar[rr]_{f_{T'}} &&  \Ct_{\AA^1}\times Y
}
\end{equation}
where $T \to T'$ is an embedding defined by a square zero ideal $J$, the left square is cartesian with $\C_T$ and $\C_{T'}$ twisted curves, and $[f_{T'}] \in \cM(Y)(T')$ and $[\n_T] \in \cM(\defY, \defsE)(T)$. Then a lift $\n_{T'}$ of $\n_T$ exists if and only if the obstruction
\[
L\n_T^*\LL_{\defZ/\Ct_{\AA^1}\times Y} \rightarrow \LL_{\C_T/\C_{T'}}\rightarrow J[1]
\] in $\Ext^1(L\n_T^*\LL_{\defZ/\Ct_{\AA^1}\times Y}, J)$ vanishes, and in this case the set of extensions is a torsor under $\Ext^0(L\n_T^*\LL_{\defZ/\Ct_{\AA^1}\times Y}, J)$.
\end{lemma}
 Observe that this lemma does not follow from \cite[III.2.2.4]{illusie71} because the cotangent complex $\LL_{\defZ/\Ct_{\AA^1\times Y}}$ is not a cotangent complex of ringed topoi, nor does it follow from \cite[Thm~1.5]{olsson06b} because our base $\Ct_{\AA^1}\times Y$ is not a scheme.
\begin{proof}
 The distinguished triangle of cotangent complexes
\[
L\n^*_T \LL_{\defZ/\Ct_{\AA^1}\times Y} \to \LL_{\C_T/\Ct_{\AA^1}\times Y} \to \LL_{\C_T/\defZ} \to 
\]
leads to a long exact sequence
\[
\Ext^0(L\n^*_T \LL_{\defZ/\Ct_{\AA^1}\times Y}, J) \to \Ext^1(\LL_{\C_T/\defZ},J) \xrightarrow{\alpha} \Ext^1(\LL_{\C_T/\Ct_{\AA^1}\times Y}, J) \xrightarrow{o} \Ext^1(L\n^*_T \LL_{\defZ/\Ct_{\AA^1}\times Y}, J).
\]
The diagram \eqref{eq:proofobs} defines commuting triangles
\[
\xymatrix{
\C_T \ar[r] \ar[d] & \C_{T'} \ar[ld]^{f_{T'}} & \mbox{and} & \C_T \ar[r] \ar[d] & \C_{T'} \ar@{-->}[ld]^{\n_{T'}} \\
\Ct_{\AA^1}\times Y &                   && \defZ 
}
\]
Applying \cite[Theorem 1.1]{olsson06b}, since all maps in the right square of \eqref{eq:proofobs} are representable, we see that the left triangle gives an element $[f_{T'}] \in \Ext^1(\LL_{\C_T/\Ct_{\AA^1}\times Y}, J)$, and likewise the right triangle gives an element $[\n_{T'}] \in \Ext^1(\LL_{\C_T/\defZ}, J)$. Under this identification, the map given by sending $[\n_{T'}]$ to $[p \circ \n_{T'}]$ is compatible with the map on Ext groups. This follows from the definition of the isomorphism in \cite[Thm~1.1]{olsson06b} and the functoriality described in \cite[III.1.2.2]{illusie71}.
Hence a lift $\n_{T'}$ exists if and only if the fiber $\alpha^{-1}([f_{T'}])$ is nonempty, if and only if the obstruction $o([f_{T'}])$ vanishes. Furthermore, the obstruction $o([f_{T'}])$ is given by the composition 
\[
L\n_T^*\LL_{\defZ/\Ct_{\AA^1}\times Y} \rightarrow \LL_{\C_T/\C_{T'}}\rightarrow J[1]
\]
where in particular the first arrow is the composition $L\n_T^*\LL_{\defZ/\Ct_{\AA^1}\times Y} \rightarrow \LL_{\C_T/\Ct_{\AA^1}\times Y} \rightarrow \LL_{\C_T/\C_{T'}}.$ This follows from the definition of the isomorphism in \cite[Thm~1.1]{olsson06b}.

Assuming $o([f_{T'}]) = 0$, we compute directly that the lifts form an $\Ext^0(L\n^*_T \LL_{\defZ/\Ct_{\AA^1}\times Y}, J)$-torsor.  
Suppose we have two lifts $\n_i$ of $\n_T$ for $i=1,2$ inducing the same $f_{T'}$. We may then view $\n_i$ as sections of $\C_{T'}\times_{\Ct_{\AA^1}\times Y}\defZ \to \C_{T'}$. We have their differences
\[
(\n^*_1 - \n^*_2) \in \Ext^0(\n^*_T\Omega_{\C_T\times_{\Ct_{\AA^1}\times Y}\defZ/\C_T}, J) = \Ext^0(\n^*_T \Omega_{\defZ/\Ct_{\AA^1}\times Y}, J) \cong \Ext^0(L\n^*_T \LL_{\defZ/\Ct_{\AA^1}\times Y}, J)
\]
by a standard calculation.
 \end{proof}

\begin{lemma}
 $R\pi_*(L\n^*\LL_{\defZ/(\Ct_{\AA^1} \times Y)}\otimes \omegabul)$ is perfect of amplitude $[-2, 0]$.
 \end{lemma}
 \begin{proof}
 Consider the composition $\defY \to E|_{Y\times\AA^1} \to Y\times\AA^1$ where the first arrow is a regular embedding. Thus by \eqref{eq:universal-target}, the morphism $\defZ \to (\Ct_{\AA^1} \times Y)$ is affine and lci. This implies that $\LL_{\defZ/(\Ct_{\AA^1} \times Y)}$ is perfect of amplitude $[-1,0]$, hence $L\n^*\LL_{\defZ/(\Ct_{\AA^1} \times Y)}\otimes \omegabul$ is perfect of amplitude $[-2,-1]$. Then (1) follows as we push forward along a family of twisted curves. 
 \end{proof}

It remains to verify that $\EE_{\cM/\Mgnt_{\AA^1}}$ is perfect in $[-1,0]$. Rotating the top of \eqref{eq:proofpot}, we obtain a distinguished triangle 
\[
R\pi_*(L\n^*\LL_{\defZ/(\Ct_{\AA^1} \times Y)}\otimes \omegabul)[-1] \longrightarrow p^*\EE_{\cM(Y)_{\AA^1}/\Mgnt_{\AA^1}} \longrightarrow \EE_{\cM/\Mgnt_{\AA^1}} \longrightarrow 
\]
Since the middle complex is perfect in $[-1,0]$ and the left one is perfect in $[-1,1]$ (this uses regularity of \eqref{eq:target-embedding}, $\EE_{\cM/\Mgnt_{\AA^1}}$ is perfect in at least $[-2,0]$. We will next prove that $\EE_{\cM/\Mgnt_{\AA^1}}$ is perfect in $[-1,0]$ by showing that $h^2(\EE^{\vee}_{\cM/\Mgnt_{\AA^1}}) = 0$.

This coherent sheaf vanishes if its fibers do. Recall the inclusion $\iota_c\colon \cM_c \to \cM(\defY, \defsE)$  of a fiber over $\AA^1$. We have
\begin{equation}\label{eq:pot-f-perfect}
\iota_c^*h^2(\EE^{\vee}_{\cM/\Mgnt_{\AA^1}}) = h^2(L\iota_c^*\EE^{\vee}_{\cM/\Mgnt_{\AA^1}} )= h^2\left(R \pi_*\big( L\iota_c^*L\n^*(\LL_{\defZ/\Ct_{\AA^1}}\otimes \omegabul)^{\vee}\big)\right),
\end{equation}
where now $L\iota_c^*$ is pullback to a closed fiber of the universal curve on $\cM(\defY, \defsE)$. 
The two equalities hold because (1) $\EE^{\vee}_{\AA^1}$ is perfect, so its derived pullback is computed by the usual pullback applied to each term; and (2) $\pi$ is flat so we may apply the tor-independent base change theorem (see e.g. \cite[Cor~4.13]{HR17}). 
The right hand side of \eqref{eq:pot-f-perfect} is precisely the second cohomology of $\EE^{\vee}_{\cM_c/\Mgnt}$ (using \eqref{eq:specialpot}). Thus, the proof of Proposition \ref{prop:defspace} is concluded by the following observation:

\begin{lemma}\label{lem:fperfect}
The fibers $\EE_{\cM_c/\Mgnt}$ in \eqref{eq:specialpot} are perfect in $[-1,0]$.
\end{lemma}
\begin{proof}
We first consider the case $c=0$. The sequence $\defZ_0 \to \Ct\times X \to \Ct$ induces a triangle of cotangent complexes
\[
L f^* \LL_{\Ct\times X/\Ct} \longrightarrow L\n^*_0\LL_{\defZ_0/\Ct} \longrightarrow L\n^*_0\LL_{\defZ_0/\Ct\times X}\longrightarrow
\]
Applying the construction of \eqref{eq:candidatepot} and rotating the triangle, we obtain a distinguished triangle
\[
R\pi_*(L\n^*_0\LL_{\defZ_0/\Ct\times X}\otimes\omegabul)[-1] \longrightarrow \EE_{\cM(X)/\Mgnt} \longrightarrow \EE_{\cM_0/\Mgnt} \longrightarrow 
\]
Since $\defY_0 = E|_X \to X$ is the projection of a vector bundle, $\defZ_0 \to \Ct\times X$ factors through a sequence of  smooth representable morphisms $\defZ_0 \to \Ct\times\defY_0 \to \Ct\times X$. Thus $\LL_{\defZ_0/\Ct\times X}$ is a vector bundle over $\defZ_0$. This implies that the left side of the above triangle is perfect in $[0,1]$. Since by assumption $\EE_{\cM(X)/\Mgnt}$ is perfect in $[-1,0]$, $\EE_{\cM_0/\Mgnt}$ is perfect in $[-1,0]$ as well. 

Replacing $X$ by $Y$, and noticing that $\EE_{\cM(Y)/\Mgnt}$ is perfect in $[-1,0]$ by assumption, the case of $c\neq 0$ can be then proved similarly.
\end{proof}

\subsection{Virtual cycles}
We now have the necessary data to define virtual cycles on $\cM=\cM(\defY, \defE)$ and its fibers $\cM_c$ using \cite[Sec~5]{BF97} (note that the need for global resolutions was removed in \cite{kresch99}). When $c \neq 0,$ by Section \ref{sec:special1} we have $\cM_c = \cM(Y, E)$. This, combined with \eqref{eq:specialpot}, yields the following corollary to Proposition \ref{prop:defspace}.
\begin{corollary}\label{cor:virtual_cycles}The moduli space $\cM$ (resp. $\cM_c$) carries a virtual fundamental class $[\cM]^\vir$ (resp. $[\cM_c]^\vir$) defined by the relative perfect obstruction theory given by the canonical morphism \eqref{eq:candidatepot}. In particular, when $c \neq 0$ we obtain a virtual fundamental class $[\cM(Y, E)]^\vir := [\cM_c]^\vir$ on $\cM(Y, E)=\cM_c$.
\end{corollary}

\begin{remark}
One can directly show that the canonical theory \eqref{eq:candidatepot}  is a perfect obstruction theory on $\cM(Y, E)$, and hence induces a virtual fundamental class $[\cM(Y, E)]^\vir$, independent of any assumptions on $s$.
\end{remark}

\section{The cosection}\label{sec:cosec}

\subsection{Relative cosection in the general setting}
We now define a cosection of the relative obstruction sheaf of the candidate obstruction theory \eqref{eq:candidatepot} on (an open substack of) the moduli of sections \eqref{eq:defsec}, in the general situation of \eqref{eq:generalsetup}. Consider the family $Z \rightarrow \C \rightarrow \mathfrak{U}$ in \eqref{eq:generalsetup}, and recall the notational conventions in \ref{sec:conventions}. Let $\H\rightarrow \C$ be the total space of the vector bundle corresponding to the dualizing sheaf $\omega_{\mathfrak{U}}$ on $\C$.
The cosection construction applies when the morphism $Z \rightarrow \C$ factors as
\[
Z \xrightarrow{\W} \H \rightarrow \C.
\]
We also assume that the base $\mathfrak{U}$ is smooth for constructing an absolute cosection in the next section.

The morphism $\W$ induces a canonical map
\[
\omega_{Sec}^\vee = L\n^*L\W^*\LL_{\H/\C} \rightarrow L\n^*\LL_{Z/\C}.
\]
Tensoring this morphism with $\omegabul_{Sec}$, applying $R\pi_*$, taking the (derived) dual, and finally composing with the canonical map $(R\pi_*(\OO_{\C_{Sec}}[1]))^\vee \xrightarrow{} \OO_{{Sec}}[-1]$, we get a \textit{cosection} for the obstruction theory as a morphism in the derived category:
\begin{equation}\label{eq:defcoseccomplex}
\sigma\colon \EE_{\Sec{Z}{\C}/\mathfrak{U}}^\vee \rightarrow (R \pi_*( \OO_{\C_{Sec}}[1]))^\vee \xrightarrow{} \OO_{{Sec}}[-1].
\end{equation}
The map $(R\pi_*(\OO_{\C_{Sec}}[1]))^\vee \xrightarrow{} \OO_{{Sec}}[-1]$ is the dual of a shift of the map $\OO_{{Sec}} \rightarrow R\pi_*(\OO_{\C_{Sec}}) $ induced via adjunction by $\pi^*\OO_{Sec} = L\pi^*\OO_{Sec} \rightarrow \OO_{\C_{Sec}}$.
The cosection $\sigma$ of the obstruction theory induces a cosection $\sigma^1$ of the relative obstruction sheaf $Ob_{\Sec{Z}{\C}/\mathfrak{U}} := h^1(\EE_{\Sec{Z}{\C}/\mathfrak{U}}^\vee)$, defined to be $h^1$ applied to \eqref{eq:defcoseccomplex}:
\begin{equation}\label{eq:defcosec}
\sigma^1\colon h^1(\EE_{\Sec{Z}{\C}/\mathfrak{U}}^\vee) \rightarrow h^1( \pi_*( \OO_{\C_{Sec}}[1]))^\vee ) \xrightarrow{\sim} \OO_{\Sec{Z}{\C}},
\end{equation}
where the second map is now an isomorphism.
The \textit{degeneracy locus} of $\sigma$ (or of $\sigma^1$) is the closed subset of $\Sec{Z}{\C}$ where the fiber of $\sigma^1$ vanishes. We denote it $\Sec{Z}{\C}(\sigma)$.

The cosection \eqref{eq:defcoseccomplex} has the useful property that it vanishes on the image of the intrinsic normal cone in $h^1/h^0(\EE^\vee_{\Sec{Z}{\C}})$. This is the content of the next lemma.
\begin{lemma}\label{lem:general-vanishing}
The map of cones induced by $\sigma \circ \phi^\vee$ is zero.
\end{lemma}
\begin{proof}
Let $\Sec{\H}{\C}$ be the moduli of sections of $\H \to \C$ over $\mathfrak{U}$. By \cite[Prop~2.2]{CL12} for orbifold curves, the morphism $\Sec{\H}{\C} \rightarrow \mathfrak{U}$ is smooth and representable by affine schemes. Composing sections of $Z \to \C$ with $\W$ induces a morphism $\mu\colon \Sec{Z}{\C} \rightarrow \Sec{\H}{\C}$.
From the left square in \eqref{eq:a.2.2-1}, we have a commuting square 
\[
\begin{tikzcd}
(\mu^*\EE_{\Sec{\H}{\C}/\mathfrak{U}})^\vee & (\EE_{\Sec{Z}{\C}/\mathfrak{U}})^\vee \arrow[l, "\hat \sigma"'] \\
(\mu^*\LL_{\Sec{\H}{\C}/\mathfrak{U}})^\vee \arrow[u] & (\LL_{\Sec{Z}{\C}/\mathfrak{U}})^\vee \arrow[l] \arrow[u, "\phi^\vee_{\Sec{Z}{\C}/\mathfrak{U}}"']
\end{tikzcd}
\]
where $\hat \sigma$ is the morphism \[\EE_{\Sec{Z}{\C}/\mathfrak{U}}^\vee \rightarrow (R \pi_*( \OO_{\C_{Sec}}[1]))^\vee\] from \eqref{eq:defcoseccomplex} followed by a quasi-isomorphism. After applying $h^1/h^0$ to this diagram, the composition $\uparrow_\leftarrow$ vanishes because
\begin{equation*}
  h^1/h^0((\mu^*\LL_{\Sec{\H}{\C}/\mathfrak{U}})^\vee) = \mathfrak{N}_{\Sec{\H}{\C}/\mathfrak{U}}=0,
\end{equation*}
where $\mathfrak{N}$ denotes the intrinsic normal sheaf, because
$\Sec{\H}{\C} \rightarrow \mathfrak{U}$ is smooth and representable by
affine schemes.
On the other hand, the desired map of cones factors through the
composition ${}^\leftarrow \uparrow$.
\end{proof}

\subsection{The family cosection}

Let $\cM = \cM(\defY, \defsE)$ with fiber $\cM_c$ over a closed point $c \in \AA^1$. 
Recall the perfect obstruction theory 
$\phi \colon \EE_{\cM/\Mgnt_{\AA^1}}\rightarrow \LL_{\cM/\Mgnt_{\AA^1}}$ and its specialization $\phi_c\colon \EE_c=\EE_{\cM_c/\Mgnt} \rightarrow\LL_{\cM_c/\Mgnt}$, see \eqref{eq:family-pot} and Proposition \ref{prop:defspace}.

In this section, we apply the general construction of \eqref{eq:defcoseccomplex} to get a cosection
\[
\sigma\colon \EE_{\cM/\Mgnt_{\AA^1}}^\vee \to \OO_{\cM}[-1].
\]
To do so, let $\varpi \rightarrow \Ct_{\AA^1}$ be the $\CC^*$-torsor of $\omega = \omega_{\Ct_{\AA^1}}$
so that $\omega = \varpi \times_{\CC^*} \CC$, hence
$\defZ = \varpi \times_{\CC^*} (\defE)^{\vee}$.
Let $\H\rightarrow \Ct_{\AA^1}$ be the total space of $\omega$.
Taking the product of \eqref{eq:family-super-potential} with $\varpi$ and
then quotienting by $\CC^*$, we see that $W$ descends to a map
\[
\W\colon \defZ \rightarrow \H
\]
over $\Ct_{\AA^1}$, so we get the relative cosection $\sigma$ above.

\subsection{Degeneracy locus}
The inclusion $X \to Y$ induces a closed immersion $\Sec{\Ct \times X}{\Ct} \hookrightarrow \Sec{\Ct \times Y}{\Ct}$ by Lemma~\ref{lem:closed}. Let $\cM(X)$ denote the intersection
\[
\cM(X) := \Sec{\Ct \times X}{\Ct} \times_{\Sec{\Ct \times Y}{\Ct}}\cM(Y).
\]
The closed embedding $X \times \AA^1 \subset \defZ$ (inclusion in $\defY$ followed by the zero section) induces a closed embedding $\cM(X)\times \AA^1 \subset \cM(\defY, \defsE)$ over $\AA^1$ by Lemma~\ref{lem:closed}.

In this section we compute the degeneracy locus $\cM(\sigma) \subset \cM$. Our main result is the following.
\begin{proposition}\label{prop:degeneracy}
As closed sets, $\cM(\sigma) = \cM(X) \times \AA^1$.
\end{proposition}
After describing the points of $\cM(\sigma)$ in more detail, we will prove each containment in Proposition \ref{prop:degeneracy} separately. The forward containment uses our assumption that $X \cap Y^0$ is smooth. 

Let $(C, n, c)$ be a closed point in $\cM$, with $C$ a twisted curve, $n\colon C \rightarrow \defZ$ a map over $\Ct_{\AA^1}$ and $c \in \AA^1$ a closed point. By definition, $(C, n, c)$ is in $\cM(\sigma)$ if the fiber $\sigma|_{(C, n, c)}$ vanishes.
We have a commuting diagram
\[
\begin{tikzcd}
Z \arrow[r, "\W_Z"] \arrow[d] & \H \arrow[d] \\
C \times \defY_c \arrow[r] & C
\end{tikzcd}
\]
where $Z = \defZ \times_{\Ct_{\AA^1}} C$, $\H\rightarrow C$ is the vector bundle of $\omega_C$, and $\W_Z$ is the restriction of $\W$.
Restricting \eqref{eq:defcoseccomplex}, we see that $\sigma|_{(C, n, c)}$ is $h^1$ of the map \[
(R \pi_*(Ln^*\LL_{Z/C}\otimes \omegabul))^\vee \rightarrow (R \pi_*\OO_C[1])^\vee
\]
followed by a quasi-isomorphism (we have used flatness of $\defZ \rightarrow \Ct_{\AA^1}$).
Both the source and target of this map are (quasi-isomorphic to) complexes of vector bundles in [0,1]. Cancelling the shifts and taking duals, we see that the fiber of the cosection vanishes if and only if
\begin{equation}
R^0\pi_*(\OO_C) \rightarrow R^0\pi_*(Ln^*\LL_{Z/C} \otimes \omega_C)
\end{equation}
is zero. 
Since base points are discrete, the sheaf $h^i(\omega_C \otimes n^*\LL_{Z/C})$ is torsion if $i\neq 0$. Using a spectral sequence, we obtain that 
\[
R^0 \pi_*(Ln^*\LL_{Z/C}\otimes \omega_C) = R^0 \pi_* \big(h^0(Ln^*\LL_{Z/C}\otimes \omega_C) \big).
\]
Therefore, $\sigma|_{(C, n, c)} = 0$ if and only if
\begin{equation}\label{eq:degeneracy1}
h^0(Ln^*(d\W_Z))\colon \omega_C^\vee \rightarrow h^0(Ln^*\LL_{Z/C})
\end{equation}
vanishes, where $d\W_Z$ denotes the canonical map of cotangent complexes induced by $\W_Z$.

The forward containment of Proposition \ref{prop:degeneracy} follows from the next lemma and Lemma \ref{lem:critical-loci} (the notation is defined in Section \ref{sec:superpotential}).
\begin{lemma}
If $(C, n, c)$ is in $\cM(\sigma)$, then $n$ factors through $\Ct_{\AA^1} \times_{\AA^1} \overline{Crit(W^0)}.$
\end{lemma}
\begin{proof}

Let $Z^0 = Z \times_{C\times \defY_c}(C\times \defY^0_c)$ and let $C^0 \subset C$ be the complement of the base locus---i.e., $C^0 = n^{-1}(\defZ_0)$. If \eqref{eq:degeneracy1} vanishes, its restriction to $C^0$ also vanishes. But the restriction $h^0(Ln^*\LL_{Z/C})|_{C^0}$ is equal to $\Omega_{Z^0/C^0}^1$ since $Z^0 \rightarrow C^0$ is smooth and Deligne--Mumford, so we know that if \eqref{eq:degeneracy1} vanishes then
\[
n^*_{C^0}d\W_{Z^0} \colon n^*_{C^0} \W_{Z^0}^*\Omega^1_C|_{C^0} \rightarrow n^*_{C^0}\Omega^1_{Z^0/C^0}
\]
vanishes; that is, $n_{C^0}$ factors through $Crit(\W_{Z^0})$.

On the other hand, recall that we have the following diagram:
\[
\begin{tikzcd}
\varpi \times (E^0)^\vee \arrow[d] \arrow[r, "id \times W^0"] & \varpi \times \CC\arrow[d]\\
Z^0  \arrow[r, "\W_{Z^0}"] & Vb(\omega_C)|_{C^0}
\end{tikzcd}
\]
where $\varpi$ is the $\CC^*$-torsor on $C$ such that $\varpi \times_{\CC^*} \CC = \omega_C$ where $\CC^*$ acts on $\CC$ via multiplication.
This diagram realizes the top row as a $\CC^*$-torsor over the bottom row; in particular the square is fibered. By functoriality of the map of cotangent sheaves, the locus $Crit(\W_{Z^0})$ is the quotient of $\varpi \times Crit(W^0)$ by $\CC^*$. We know from Lemma \ref{lem:critical-loci} that $Crit(W^0)$ is $\CC^*$-invariant, so this quotient is precisely $C \times Crit(W^0)$. Hence, by the previous paragraph, $n_{C^0}$ factors through $C\times Crit(W^0)$.
 Taking closure, topology forces $n$ to factor through $C \times \overline{Crit(W^0)}$.
\end{proof}

For the backwards containment of Proposition \ref{prop:degeneracy}, we claim that if $(C, n, c)$ is a closed point in $\cM$ such that $n$ factors through $\Ct_{\AA^1}\times X$, then the map $h^0(Ln^*(d\W_Z))$ of \eqref{eq:degeneracy1} is zero. The claim is a consequence of the following lemma.

\begin{lemma}\label{lem:deriv_vanishes}
  Let $Y \to S$ be a morphism of algebraic stacks, let $L$ be a line
  bundle on $S$, and let $E$ be a vector bundle on $Y$ with a section
  $s$ with zero locus $X$.
  Let $W\colon L \otimes E^\vee \to L$ be the morphism induced by $s$.
  Then $dW|_X\colon L^\vee \cong \LL_{L/S}|_X \to \LL_{Y/S}|_X$ is the
  zero morphism.
\end{lemma}

\begin{proof}
As a first example, let $Y_1 = \CC^n$ with coordinates
  $x_1, \dotsc, x_n$, let $S_1$ be a point and $L_1 = \CC$ be trivial, set
  $E_1 = Y_1 \times Y_1$ and $s_1 = (x_1, \dotsc, x_n)$, so that
  $X_1$ is the origin.
  Then $E_1^\vee\otimes L_1  =  Y_1 \times Y_1^\vee\otimes L_1 \cong \CC^n \times \CC^n$ with coordinates
  $x_1, \dotsc, x_n, p_1, \dotsc, p_n$, and $W_1\colon E_1^\vee\otimes L_1 \to L_1$
  is given by
  \begin{equation*}
    W_1(x_1, \dotsc, x_n, p_1, \dotsc, p_n)
    = p_1 x_1 + \dotsb + p_n x_n
    = \langle p, s\rangle,
  \end{equation*}
  where $p = (p_1, \dotsc, p_n)$.
  It follows that
  $dW_1\colon \LL_{L_1/S_1}|_{E_1^\vee} = \OO_{E_1^\vee} \to \Omega_{E_1^\vee} =
  \LL_{L_1 \otimes E_1^\vee/S_1}$ is given by
  \begin{equation*}
    dW_1 = \sum_{i = 1}^n (p_i dx_i + x_i dp_i)
    = \langle p, ds\rangle + \langle s, dp\rangle,
  \end{equation*}
  and hence $dW_1|_{X_1}=0$ since $p$ and $s$ vanish on $X_1$.
  
  Now consider a second example that is the quotient of the above example. We let $\CC^*$ act on $L_1 \rightarrow S_1$ by scaling $L_1 \cong \CC$, and we let $G = GL(n)\times \CC^*$ act diagonally on $E_1 \rightarrow Y_1$ via the standard representation of $GL(n)$ on $Y_1$ and the trivial action of $\CC^*$. In the resulting example we have $Y_2 = [Y_1/G]$, $S_2 = B\CC^*$, $L_2 = [L_1/\CC^*]$, and $E_2 = [(Y_1 \times Y_1)/G].$ The section $s_2$ is still the diagonal one and its vanishing locus $X_2$ is $BG \subset Y_2$. Note that $E_2^\vee \otimes L_2 = [(Y_1 \times Y_1^\vee\otimes L_1)/G]$.
  We wish to show that the map of cotangent complexes $dW_2$ for the diagram
  \begin{equation}\label{eq:deriv_vanishes1}
    \begin{tikzcd}
      E_2^\vee \otimes L_2 \arrow[r, "W_2"] \arrow[d] & {[L_1/\CC^*]} \arrow[d] \\
      B\CC^* \arrow[r, equal] & B\CC^*
    \end{tikzcd}
  \end{equation}
  vanishes after pulling back to $X_2$. For this it suffices to check that the pull-back $(d W_2)|_{X_1}$
  of $(dW_2)|_{X_2}$ along $X_1=pt \to X_2$ vanishes. Consider the commutative diagram of solid arrows where the bottom
  sequence is the pullback along $X_1 \rightarrow E_1^\vee \otimes L_1$ of triangle associated to the quotient map
  $E_1^{\vee}\otimes L_1 \to E_2^{\vee}\otimes L_2$:
  \[
  \xymatrix{
  && \LL_{L_1}|_{X_1} \ar@{-->}[lld]_{\varphi} \ar[d]^{(d W_2)|_{X_1}} \ar[rrd]^{(d W_1)|_{X_1} = 0} && & \\
  \LL_{(E_1^{\vee}\otimes L_1)/(E_2^{\vee}\otimes L_2)}|_{X_1}[-1] \ar[rr] && \LL_{E_2^{\vee}\otimes L_2}|_{X_1} \ar[rr] && \LL_{E_1^{\vee}\otimes L_1}|_{X_1} \ar[r] &
  }
  \]
  The vanishing $(d W_1)|_{X_1} = 0$ implies that $(d W_2)|_{X_1}$
  factors through $\varphi$.
  Since both $\LL_{L_1}$ and
  $\LL_{(E_1^{\vee}\otimes L_1)/(E_2^{\vee}\otimes L_2)}|_{X_1}$ are
  locally free in degree $0$, we observe $\varphi = 0$, hence
  $(d W_2)|_{X_1} = 0$.
  
  Finally, we claim that the general case factors through the second example. Indeed, given $Y, S, E, L$, and $s$ as in the statement of Lemma \ref{lem:deriv_vanishes}, we get a commuting cube whose left and right sides are fibered:
  \[
  \begin{tikzcd}[row sep=tiny, column sep=tiny]
& E_2^\vee \otimes L_2  \arrow[rr, "W_2"] \arrow[dd] & & L_2  \arrow[dd] \\
E^\vee \otimes L \arrow[rr, crossing over, "\quad \quad W"]\arrow[ur] \arrow[dd] & & L \arrow[ur]\\
& Y_2  \arrow[rr] & & S_2  \\
Y \arrow[rr]\arrow[ur] & & S \arrow[ur]\arrow[from=uu, crossing over]\\
\end{tikzcd}
\]
  The map $Y \rightarrow Y_2 = [Y_1/G] = [Y_1/GL(n)]\times B\CC^*$ is induced the vector bundle $E$ and its section $s$ (which yield a map $Y \rightarrow [Y_1/GL(n)]$) and the pullback of $L$ (which yields the map $Y \rightarrow B\CC^*$). 
  From the functoriality of the cotangent complex (see e.g. \cite[Lem~2.2.12]{webb-thesis} we get a commuting square
  \[
  \begin{tikzcd}
  \LL_{L_2/S_2} \arrow[r, "dW_2"] \arrow[d, "\sim"] & \LL_{E_2^\vee\otimes L_2/S_2} \arrow[d] \\
  \LL_{L/S} \arrow[r, "dW"] & \LL_{E^\vee \otimes L/S}
  \end{tikzcd}
  \]
  Since $dW_2$ vanishes, so does $dW$.
\end{proof}

\subsection{Specialization}
We can specialize $\sigma$ to obtain relative cosections $\sigma_c\colon \EE_{\cM_c/\Mgnt}^\vee \rightarrow \OO_{\cM_c}[-1]$. Let $\iota_c:\cM_c\rightarrow \cM$ be the inclusion.
\begin{definition}
The \textit{specialization} $\sigma_c$ of $\sigma$ is the restriction
\[
\sigma_c \colon L\iota_c^*\EE_{\cM/\Mgnt_{\AA^1}}^\vee \xrightarrow{L\iota_c^*\sigma} L\iota_c^*(R \pi_*(\OO_{\C_{\cM}}[1]))^\vee \xrightarrow{} L\iota_c^*(\OO_{\cM}[-1]) = \OO_{\cM_c}[-1].
\]
Applying $h^1$ we get a cosection $\sigma_c^1$ of $Ob_{\cM_c/\Mgnt}$.
\end{definition}

We could also define a cosection for $\EE^\vee_{\cM_c/\Mgnt}$ by applying the construction \eqref{eq:defcosec} to the restriction $\W_c\colon \defZ_c \rightarrow \H_c$ of $\W$, using the identification in \eqref{eq:specialpot}. Call this alternate cosection $\rho$. 

\begin{lemma} The cosections $\sigma_c$ and $\rho$ are canonically isomorphic.
\end{lemma}
\begin{proof}
By the functoriality of the cotangent complex we have a commuting square
\[
\begin{tikzcd}
L\iota_c^*L\W^*\LL_{\H/\Ct_{\AA^1}} \arrow[r] \arrow[d,"\sim"] &  \arrow[d, "\sim"] L\iota_c^*\LL_{\defZ/\Ct_{\AA^1}} \\
L\W_c^*\LL_{\H_c/\Ct} \arrow[r] & \LL_{\defZ_c/\Ct}
\end{tikzcd}
\]
where all arrows are canonical morphisms of cotangent complexes except for the left arrow, which differs from a canonical arrow by a quasi-isomorphism. The two vertical arrows are isomorphisms because $\defZ$ and $\H$ are both flat over $\Ct$.
After applying $L\n_c^*$, tensoring with $\omegabul$, applying $R (\pi_c)_*$ and dualizing, we obtain the square
\[
\begin{tikzcd}
L\iota_c^*(R\pi_*\OO_{\C_{\cM}}[1])^\vee  &   \EE_{\cM_c/\Mgnt}^\vee \arrow[l] \\
(R(\pi_c)_*\OO_{\C_{\cM_c}}[1])^\vee \arrow[u, "\sim"]& (R(\pi_c)_* (L\n_c^* \LL_{\defZ_c/\Ct}\otimes \omegabul))^\vee\arrow[l] \arrow[u, "\sim"]
\end{tikzcd}
\]
where the right vertical arrow is the dual of the left vertical arrow in \eqref{eq:specialpot}.
Composing with the adjunction homomorphisms, the top horizontal arrow
becomes $\sigma_c$, while the bottom horizontal arrow becomes
$\rho$.
\end{proof}

This together with Lemma \ref{lem:general-vanishing} implies that the relative cosections $\sigma_c$ have the property that the map of cones induced by $\sigma_c \circ \phi_c^\vee$ is zero. Moreover, combining this result with Lemma \ref{lem:computecosec} and Proposition \ref{prop:degeneracy} we have the following.
\begin{corollary}\label{cor:specialcosec}
The specialization $\sigma_c$ for $c\neq 0$ is induced by the section $s \in \Gamma(Y, \E)$, and $\sigma_0$ is induced by the tautological section of $\E|_{N_{X/Y}} \simeq \E|_{E_X}$. In either case, the degeneracy locus $\cM_c(\sigma_c)$ of the specialized cosection is equal to $\cM(X) \subset \cM_c$.
\end{corollary}

\subsection{Cosection localized virtual cycles}

We have the necessary data to define cosection localized virtual cycles on $\cM = \cM(\defY, \defsE)$ and its fibers $\cM_c$. Specifically, set
\[
[\cM]^{\vir}_{\sigma} := 0^!_{\sigma}([\mathfrak{C}_{\cM/\Mgnt_{\AA^1}}])\in A_*(\cM(X)\times \AA^1) \quad \quad \quad [\cM_c]^\vir_{\sigma_c}:= 0^!_{\sigma_c}([\mathfrak{C}_{\cM_c/\Mgnt}]) \in A_*(\cM(X))
\]
where the cosection localized Gysin maps $0^!_\bullet$ are defined in \cite[(3.5)]{KiLi13}. By Proposition \ref{prop:cos-loc-func} these classes agree with the ones defined in \cite{KiLi13} using the corresponding absolute obstruction theory and cosection, so that in particular
by \cite[Theorem~5.2]{KiLi13}, we have 
\begin{equation}\label{eq:proof-of-1}
\iota_* [\cM]^\vir_\sigma = [\cM]^\vir  \quad \quad \quad \iota_*[\cM_c]^\vir_{\sigma_c}=[\cM_c]^\vir,
\end{equation}
where the virtual classes $[\cM]^\vir$ and $[\cM_c]^\vir$ were defined in Corollary \ref{cor:virtual_cycles}. Recall that by definition, $[\cM(Y, E)]^\vir$ is equal to the class $[\cM_c]^\vir$ for $c \neq 0$.

These cosection localized virtual cycles are related by the following lemma, which is \cite[Theorem~4.6]{CL18}. Our proof below follows exactly that of (loc. cit.); we include it for completeness.
\begin{lemma}\label{lem:comparevc}
The constructed virtual cycles satisfy
\[
\iota_c^![\cM]^\vir_{\sigma} = [\cM_c]^\vir_{\sigma_c}
\]
where $\iota_c^!$ is the Gysin map for the regular embedding $\iota_c\colon \cM_c \rightarrow \cM$.
\end{lemma}
\begin{proof}
This argument follows that of \cite[Theorem~4.6]{CL18}.
We have morphisms of algebraic stacks $\cM \rightarrow \Mgnt\times \AA^1 \rightarrow \Mgnt$ and a relative perfect obstruction theory $\phi\colon \EE_{\cM/\Mgnt \times \AA^1} \rightarrow \LL_{\cM/\Mgnt\times \AA^1}$ with a cosection $\sigma$. By Proposition \ref{prop:cos-loc-func} there is a relative perfect obstruction theory $\EE_{\cM/\Mgnt} \rightarrow \LL_{\cM/\Mgnt}$ on $\cM$ with a cosection $\rho$ such that $\rho$ descends to an absolute cosection and $[\cM]_\rho^{\vir}=[\cM]_{\sigma}^{\vir}$. Moreover, restricting the morphism of distinguished triangles \eqref{eq:cos-loc-func1} to $\cM_c \subset \cM$, we obtain the top two rows of the following diagram of distinguished triangles:
\begin{equation}\label{eq:comparevc}
\begin{tikzcd}
\EE_{\cM/\Mgnt}|_{\cM_c} \arrow[r] \arrow[d] & \EE_{\cM/\Mgnt\times\AA^1}|_{\cM_c} \arrow[r] \arrow[d] & q^*\LL_{\Mgnt \times \AA^1/\Mgnt}[1]|_{\cM_c} \arrow[r]\arrow[d, equal] & {}\\
\LL_{\cM/\Mgnt}|_{\cM_c} \arrow[r] \arrow[d, equal] & \LL_{\cM/\Mgnt\times \AA^1}|_{\cM_c} \arrow[r] \arrow[d] & q^*\LL_{\Mgnt \times \AA^1/\Mgnt}[1]|_{\cM_c} \arrow[r] \arrow[d, dashrightarrow]& {}\\
\LL_{\cM/\Mgnt}|_{\cM_c} \arrow[r]& \LL_{\cM_c/\Mgnt} \arrow[r] & \LL_{\cM_c/\cM} \arrow[r] & {}
\end{tikzcd}
\end{equation}
the bottom row is the canonical triangle and the dotted arrow is induced by the mapping cone axiom. The middle column is precisely the specialization $\phi_c$ in \eqref{eq:specialpot} which by Proposition \ref{prop:defspace} is isomorphic to the perfect obstruction theory on $\cM_c$. The morphisms between the top and bottom rows of \eqref{eq:comparevc} are the compatibility required to apply \cite[Thm~5.2]{KiLi13}. By that result, $\iota^![\cM]^{\vir}_\rho = [\cM_c]^\vir_{\sigma_c}$ as desired. 
\end{proof}

\section{Calculation in the special fiber}
\label{sec:specialfiber}

In this section we investigate the special fiber $\cM_0$ of $\cM = \cM(\defY, \defsE)$. Recall from \eqref{eq:specialm0} that
\[
\cM_0 = \cM(X) \times_{\Sec{\Ct\times X}{\Ct}} \Sec{\defZ_0}{\Ct} \quad \quad \text{with} \quad \defZ_0 = Vb_{\Ct \times X}(\E_X \oplus \omega \otimes \E_X^\vee).
\]
Recall from Proposition \ref{prop:defspace} and \eqref{eq:specialpot} that $\cM_0$ has a canonical relative perfect obstruction theory
\[
\phi_0\colon R \pi_*(L\n^*\LL_{\defZ_0/\Ct}\otimes \omegabul) \rightarrow \LL_{\cM_0/\Mgnt}
\]
where $\n\colon \C_{\cM_0} \rightarrow \defZ_0$ is the universal section, and by Corollary \ref{cor:specialcosec} it carries a cosection $\sigma_0$ induced by the tautological section of $\E|_{E_X}$. The degeneracy locus of the cosection is contained in $\cM(X) \subset \cM_0$, embedded via the zero section. The goal of this section is to prove the following.

\begin{theorem}\label{thm:specialfiber}
  We have an equality of virtual classes
  \[
    [\cM_0]^{\vir}_{\sigma_0} = (-1)^{\chi(E, \cM(X))}[\cM(X)]^{\vir}\;\text{in}\;A_*(\cM(X)).
  \]
  Here, $(-1)^{\chi(E, \cM(X))} \in A^0(\cM(X))$ is the locally
  constant function defined via \eqref{eq:virtual-rank-intro}.
\end{theorem}
\begin{remark}\label{rmk:virtual-rank}
  In Theorem \ref{thm:specialfiber}, the function $\chi(E, \cM(X))$ is
  constant on each connected component of $\cM(X)$.
  If $\{\cM_i\}_{i=1}^m$ are the connected components of $\cM(X)$,
  then we define 
  \[(-1)^{\chi(E, \cM(X))} [\cM(X)]^{\vir} = \bigoplus_{i=1}^m (-1)^{\chi(E, \cM_i)}[\cM_i]^{\vir}\]
  as a class in $A_*(\cM(X)) = \bigoplus A_*(\cM_i).$

  Explicitly, by Riemann--Roch for twisted curves
  \cite[Theorem~7.2.1]{AGV08}
  \begin{equation}
    \label{eq:virtual-rank}
    \chi(E, \cM_i)
    = \rk(\E) (1 - g) + \int_\beta c_1(\E) - \sum_{j = 1}^n \mathrm{age}_j(\E)
  \end{equation}
  where $\mathrm{age}_j(\E)$ is the age of $f^* \E$ at the $j$th
  marking, which is constant on $\cM_i$.
\end{remark}

The proof of Theorem \ref{thm:specialfiber} uses cosection localized
torus localization \cite[Thm~3.4]{CKL17}.
We recall two definitions from there.
First, if $X$ is a Deligne--Mumford stack with a $\CC^*$-action, the
\textit{fixed locus} $X^{\CC^*}$ is defined \'etale-locally as
follows.
Let $\Spec(A) \rightarrow X$ be an \'etale affine chart, equivariant
after reparametrization (such charts exist by \cite[Thm~4.3]{AHR19});
then $A$ is a $\ZZ$-graded ring and in this chart the fixed locus is
cut out by the ideal generated by elements of positive weight.
One can check that a closed point $x \in X(\CC)$ is in the fixed locus
if and only if $tx$ is isomorphic to $x$ for every $t \in \CC^*$ (see
for example \cite[Prop~5.23]{AHR19}).

Second, if $E$ is a perfect complex of sheaves on $X$, we say that $E = A \oplus B$ is a decomposition into fixed and moving parts if given an equivariant \'etale chart $U \rightarrow X$ and a quasi-isomorphism $f\colon E|_U\rightarrow F$ with $F$ a bounded complex of vector bundles, $f$ induces quasi-isomorphisms $A \rightarrow F^{\fix}$ and $B \rightarrow F^{\mv}$ to the fixed and moving parts of $F$, respectively.

In preparation to apply torus localization, we prove the following lemmas. The proof of the first was explained to us by Bumsig Kim and Jeongseok Oh (see also \cite[Footnote~1,~p.~35]{CFGKS}).
\begin{lemma}\label{lem:global-res}
Let $\cM$ be a Deligne-Mumford stack of finite type and let $\pi:\C \rightarrow \cM$ be a family of prestable twisted curves as in \cite{AV02}. If there exists a $\pi$-relatively ample line bundle on $\C$, then for any vector bundle $E$ on $\C$, the complex $R\pi_*E$ is globally isomorphic to a 2-term complex of vector bundles. 
\end{lemma}
\begin{proof}
Factor $\pi$ as the composition $\C \xrightarrow{p} C \xrightarrow{q} \cM$ where $C$ is the relative coarse moduli space \cite[Thm~3.1]{AOV11}. By assumption, there exists a $q$-relatively ample line bundle $\OO(1)$ on $C$, meaning that for any affine scheme $U \rightarrow \cM$ the pullback of $\OO(1)$ to $C \times_{\cM} U$ is ample. Let $\OO(n)$ denote $\OO(1)^{\otimes n}$.

Choose $n \gg 0$ such that
\begin{enumerate}
\item $R^1q_*\OO(n) = 0$,
\item $R^1q_*(p_*(E) \otimes \OO(n)) = 0$, and
\item $q^*q_*\OO(n) \rightarrow \OO(n)$ is surjective.
\end{enumerate}
Indeed, this is possible when $\cM$ is an affine scheme by \cite[Thm~III.8.8]{hartshorne}. We cover a general $\cM$ with finitely many affine schemes and take the maximum value of the respective $n$'s.

Now points (3) and (1) yield a surjection of vector bundles $\pi^*q_*\OO(n) \rightarrow p^*\OO(n)$. Let $K$ denote the kernel of this map, also a vector bundle. Taking the dual of the resulting short exact sequence and tensoring with $p^*\OO(n) \otimes E$ yields an exact sequence of vector bundles
\begin{equation}\label{eq:recent}
0 \rightarrow E \rightarrow (\pi^*q_*\OO(n))^\vee \otimes p^*\OO(n) \otimes E \rightarrow K^\vee \otimes p^*\OO(n) \otimes E \rightarrow 0.
\end{equation}

Using the projection formula twice, we compute
\begin{align*}
R^1\pi_*\Big( (\pi^*q_*\OO(n))^\vee \otimes p^*\OO(n) \otimes E \Big)&= H^1\Big( R\pi_*\big( (\pi^*q_*\OO(n))^\vee \otimes p^*\OO(n) \otimes E \big) \Big)\\
&= H^1\Big(  (q_*\OO(n))^\vee \otimes^L Rq_*(\OO(n) \otimes^L Rp_*E)  \Big)
\end{align*}
Since $p_*$ is exact (see e.g. \cite[Prop~11.3.4]{olsson}) we may replace $Rp_*E$ with $p_*E$; in particular this is a coherent sheaf. Moreover since $\OO(n)$ and $q_*\OO(n)$ are vector bundles (by (1)) we may replace the derived tensor products $\otimes^L$ with the usual one. Since tensoring with $(q_*\OO(n))^\vee$ commutes with taking cohomology, we see that (2) above implies that the sheaf $R^1\pi_*\Big( (\pi^*q_*\OO(n))^\vee \otimes p^*\OO(n) \otimes E \Big)=0$. Now the long exact sequence for $R\pi_*$ applied to \eqref{eq:recent} produces the desired resolution of $R\pi_* E$.
\end{proof}

\begin{lemma}\label{lem:decomposition}
The moduli space $\cM_0$ has a $\CC^*$-action such that
\begin{enumerate}
\item As closed sets, the fixed locus $\F$ of the action is naturally identified with $\cM(X)$
\item There is a decomposition into fixed and moving parts 
\[\EE_{\cM_0/\Mgnt}|_\F = \EE^{\fix}  \oplus \EE^{\mv}\]
where $\EE^{\fix} = \EE_{\cM(X)/\Mgnt}$ and $\EE^{\mv} = R \pi_*(f^*\E^\vee \otimes \omegabul \oplus f^*\E[1]),$ with $\pi\colon \C_{\cM(X)} \rightarrow \cM(X)$ the universal curve and $f\colon \C_{\cM(X)} \rightarrow X$ the universal map. 
\end{enumerate}
\end{lemma}
\begin{proof}
The torus $\CC^*$ acts on $\defZ_0$ by scaling its fibers over $\Ct \times X$: let it act with weight $1$ on $E$ and weight $-1$ on $E^\vee \otimes \omega$. The fixed locus for this action is $X$.

The $\CC^*$-action on $\defZ_0$ induces one on $\cM_0$; let $\F$ be the fixed locus. Suppose $(n\colon C \rightarrow \defZ_0)$ is a closed point of $\F$ for some twisted curve $C$. Then $n$ factors through $\Ct \times X \subset \defZ_0$ for purely topological reasons as follows. Let $\overline{\PP}(\defZ_0)$ be the projective closure of $\defZ_0$ as a vector bundle over $\Ct \times X$. Because $(C, n)$ is fixed, if $t \in \CC^*$ is any closed point, then $(C, n)$ is isomorphic to $(C, t\circ n)$ as points of $\cM_0$, where $t$ also denotes the function on $\defZ_0$ defined by $t \in \CC^*$. In particular, the image $n(C)$ is invariant under the $\CC^*$-action. Invariance implies that if $n(C)$ contains some point of $\defZ_0$ not in $\Ct \times X$, then it contains the entire fiber $F$ of $\defZ_0$ containing this point. Then the composition 
\[\tilde n\colon C \xrightarrow{n} \defZ_0 \rightarrow \overline{\PP}(\defZ_0)\] 
is proper and hence closed, so it is surjective onto its image which must contain the projective closure of $F$. This is a contradiction since $\tilde n$ factors through $\defZ_0$. Conversely, if $n$ factors through $\Ct \times X$ then $(n\colon C \rightarrow \defZ_0)$ is $\CC^*$-fixed. This proves (1).

Now let $\iota\colon \Ct \times X \rightarrow \defZ_0$ be the inclusion of the zero section. The $\CC^*$-equivariant sequence of maps $\defZ_0 \xrightarrow{} \Ct \times X \rightarrow \Ct$ produces a distinguished triangle of $\CC^*$-equivariant tangent complexes which we can restrict to $\Ct \times X$, obtaining
\begin{equation}\label{eq:splitting}
\LL_{\Ct \times X/\Ct} \rightarrow L\iota^*\LL_{\defZ_0/\Ct} \rightarrow L\iota^*\LL_{\defZ_0/\Ct \times X} \rightarrow .
\end{equation}
In fact, this triangle splits by the canonical map $L\iota^*\LL_{\defZ_0/\Ct} \rightarrow \LL_{\Ct \times X/ \Ct}$ induced by $\iota$. Let $\pi_\F\colon \Ct_\F \rightarrow \F$ and $\n_\F\colon \C_\F \rightarrow \defZ_0$ be the restrictions of $\pi$ and $\n$ to $\C_\F$, the restriction of the universal curve to $\F$. We have seen that $\n_\F = \iota \circ f$ for a morphism $f\colon \C_\F \rightarrow\Ct \times X$. Because $\pi$ is flat we have
\[
\EE_{\cM_0/\Mgnt}|_\F = \left(R \pi_*(L\n^*\LL_{\defZ_0/\Ct}\otimes \omegabul) \right)|_\F = R (\pi_\F)_*(Lf^*L\iota^*\LL_{\defZ_0/\Ct}\otimes \omegabul).
\]
Hence, applying $R (\pi_\F)_*(Lf^*(\bullet) \otimes \omegabul)$ to the splitting sequence \eqref{eq:splitting} we obtain
\[
\EE_{\cM_0/\Mgnt}|_\F = R (\pi_\F)_*(f^*\E^\vee \otimes \omegabul \oplus f^*\E[1])\oplus \EE_{\cM(X)/\Mgnt},
\]
where $\EE_{\cM(X)/\Mgnt}$ is defined in \eqref{eq:initialpotX} and $\LL_{\defZ_0/\Ct \times X}$ was computed using \eqref{eq:specialm0}.
Since local sections of $\E$ are all scaled with weight $1$ by the $\CC^*$-action, local sections of $f^*\E$ and $\omega \otimes f^*\E^\vee$ are as well with weight $\pm 1$, and in particular $R (\pi_\F)_*(f^*\E[1])$ and $R (\pi_\F)_*(\omegabul \otimes f^*\E^\vee)$ have pure weights $\pm 1$. Likewise, since $\Ct \times X$ and hence $\LL_{\Ct \times X/\Ct}$ is $\CC^*$-fixed, the cohomology sheaves of $\EE_{\cM(X)/\Mgnt}$ are $\CC^*$-fixed. 
This proves (2).
\end{proof}

\begin{proof}[Proof of Theorem \ref{thm:specialfiber}]
To apply \cite[Thm~3.4]{CKL17} we must use an absolute perfect obstruction theory $\phi_{abs}\colon \EE_{\cM_0} \rightarrow \LL_{\cM_0}.$ It may be defined by the morphism of distinguished triangles
\begin{equation}\label{eq:specialfiber1}
\begin{tikzcd}
\EE_{\cM_0} \arrow[r] \arrow[d, "\phi_{abs}"] & \EE_{\cM_0/\Mgnt} \arrow[r] \arrow[d, "\phi"] & q^*\LL_{\Mgnt}[1] \arrow[r] \arrow[d, equal] & {}\\
\LL_{\cM_0} \arrow[r] & \LL_{\cM_0/\Mgnt} \arrow[r] & q^*\LL_{\Mgnt}[1] \arrow[r] &{}
\end{tikzcd}
\end{equation}
where $q\colon \cM_0 \rightarrow \Mgnt$ is the projection. By Proposition \ref{prop:cos-loc-func}, $\phi_{abs}$ is a perfect obstruction theory and it carries a cosection $\sigma_{abs}$, and the induced cosection localized virtual fundamental class is equal to $[\cM_0]^{\vir}_\sigma$. This morphism of distinguished triangles is equivariant (i.e., pulled back from a morphism of DTs on $[\cM_0/\CC^*]$) by Lemma~\ref{lem:eqabsolutepot}. Moreover, note that the original cosection $\sigma_{\cM_0/\Mgnt}$ was equivariant for the $\CC^*$-action scaling $\E$, since it was induced by an equivariant section; so $\sigma_{abs}$ is equivariant as well.

The splitting of $\EE_{\cM_0/\Mgnt}$ in Lemma \ref{lem:decomposition} induces a splitting of $\EE_{\cM_0}$ as follows. There is a commuting diagram where all rows and columns are distinguished:
\[
\begin{tikzcd}
\EE_{\cM(X)} \arrow[r] \arrow[d] & \EE_{\cM(X)/\Mgnt} \arrow[r] \arrow[d] & q_\F^*\LL_{\Mgnt} \arrow[r] \arrow[d, equal] & {}\\
\EE_{\cM_0}|_\F \arrow[r] \arrow[d] & \EE_{\cM_0/\Mgnt}|_\F \arrow[r] \arrow[d] & q_\F^*\LL_{\Mgnt} \arrow[r] \arrow[d] &{}\\
\EE^{\mv} \arrow[r, "\sim"] \arrow[d] & \EE^{\mv} \arrow[r] \arrow[d] & 0 \arrow[r] \arrow[d] & {}\\
{} & {} & {} &
\end{tikzcd}
\]
Here $\EE_{\cM(X)}$ is the absolute obstruction theory for $\cM(X)$, defined as in \eqref{eq:specialfiber1}, and $q_\F$ is the restriction of $q$ to $\F $. To obtain this diagram, begin with the middle horizontal and vertical triangle, and observe first that the top right square commutes. Note that the splitting $\EE_{\cM_0/\Mgnt}|_\F \rightarrow \EE_{\cM(X)/\Mgnt}$ and equality on $q_\F^*\LL_{\Mgnt}$ induce a splitting of the leftmost column. We conclude that
\[
  \EE_{\cM_0}|_\F = \EE_{\cM(X)} \oplus \EE^{\mv}
\]
is a decomposition into fixed and moving parts, respectively. 
By Lemma \ref{lem:global-res} and our assumptions in Section \ref{sec:introdefs},
we may find a resolution
$R \pi_*(f^*\E) = [\E_0\rightarrow \E_1]$ where $\E_0$ and $\E_1$ are
locally free of ranks $r_0$ and $r_1$, respectively, and of
$\CC^*$-weight $1$. In particular, $\EE^{\mv}$ has a global resolution and we may apply \cite[Thm~3.5]{CKL17}, obtaining
\[
[\cM_0]^{\vir}_{\sigma_0} = \frac{[\cM(X)]^{\vir}}{e((R \pi_*(f^*\E^\vee \otimes \omegabul \oplus f^*\E[1]))^\vee)} \in A^{\CC^*}(\cM_0(\sigma))\otimes_{\QQ[t]}\QQ[t,t^{-1}],
\]
where the euler class is the $\CC^*$-equivariant one, we have identified the virtual class of the fixed locus using Lemma \ref{lem:decomposition}, and we have remembered that the degeneracy locus in $\cM_0$ is contained in $\cM(X) \subset \cM_0$.
By Serre duality,
\[
R \pi_*(f^*\E^\vee\otimes \omegabul) \cong (R \pi_*(f^*\E))^\vee.
\]
Therefore,
\begin{equation*}
  e((R \pi_*(f^*\E^\vee \otimes \omegabul \oplus f^*\E[1]))^\vee)
  = \frac{e(R \pi_*(f^*\E))}{e((R \pi_*(f^*\E))^\vee)}.
\end{equation*}
Using the resolution $R\pi_*(f^*\E) = [\E_0 \rightarrow \E_1]$, we can write
\[
  e((R \pi_*(f^*\E^\vee \otimes \omegabul \oplus f^*\E[1]))^\vee)
  = \frac{e(\E_0)}{e(\E_1)}\frac{e(\E_1^\vee)}{e(\E_0^\vee)}
  = \frac{e(\E_0)}{e(\E_1)}\frac{e(\E_1)}{e(\E_0)} (-1)^{r_1-r_0}
  = (-1)^{r_0-r_1}.
\]
Noting that by definition $r_0 - r_1$ is the value of
$\chi(E, \cM(X))$ on the component of $\cM(X)$ under consideration,
this finishes the proof of the theorem.
\end{proof}

\begin{proof}[Proof of Theorem \ref{thm:mainthm}]
The Gysin maps $\iota_c^!\colon A_*(\cM(X) \times \AA^1) \rightarrow A_*(\cM(X))$ are independent of $c$, so fixing some $c \neq 0$ we have
\begin{equation}\label{eq:almostdone}
\iota_c^![\cM]^{\vir}_{\sigma} = \iota_0^![\cM]^{\vir}_{\sigma} \quad \quad \in A_*(\cM(X)).
\end{equation}
By Lemma \ref{lem:comparevc} the left hand side is $[\cM_c]^\vir_{\sigma_c} = [\cM(Y, E)]^{\vir}_{\sigma_c}$; combined with \eqref{eq:proof-of-1} this proves \eqref{eq:thma}. By Lemma \ref{lem:comparevc} and Theorem \ref{thm:specialfiber} the right hand side of \eqref{eq:almostdone} is $(-1)^{\chi(E, \cM(X))}[\cM(X)]^{\vir}$, proving \eqref{eq:thmb}.
\end{proof}

\section{Applications}\label{sec:applications}
We explain how Theorem \ref{thm:mainthm} applies in various situations, and we relate it to existing constructions of the moduli of p-fields.
\subsection{Application to stable maps}\label{sec:stablemaps}

Let $Y$ be a smooth projective Deligne--Mumford stack.
Choose a vector bundle $E$ on $Y$ and a regular section whose zero
locus $X$ is smooth.
Fix nonnegative integers $g,n$ and a class
$\beta \in H_2(\underline{Y}),$ where $\underline{Y}$ denotes the
coarse moduli space of $Y$.

\begin{proof}[Proof of Corollary \ref{cor:stablemaps}]
  Let $\iota\colon \underline{X} \rightarrow \underline{Y}$ denote the
  inclusion, and let $\cM(Y)$ be the moduli space of stable twisted maps
  $\Mbar_{g,n}(Y, \beta)$.
  A priori, $\Mbar_{g,n}(Y, \beta)$ is an open substack of
  $\Hom_{\Mgnt}(\Ct, Y \times \Mgnt)$; by Lemma \ref{lem:fibered_sections}
  canonical map
  $\Sec{\Ct \times Y}{\Ct} \rightarrow \Hom_{\Mgnt}(\Ct, Y \times
  \Mgnt)$ is an equivalence.
  Moreover the morphism \eqref{eq:candidatepot} is the usual
  obstruction theory on $\Hom_{\Mgnt}(\Ct, Y \times \Mgnt)$, so by
 \cite[Section 4.5]{AGV08} its restriction to $\Mbar_{g,n}(Y, \beta)$ is
  perfect.
  The moduli space $\cM(X)$ is a disjoint union over
  $\beta' \in H_2(\underline{X})$ with
  $\iota_* \beta' = \beta$:
  \begin{equation}\label{eq:stablemaps2}
    \cM(X) = \Mbar_{g,n}(X, \beta) :=\bigsqcup_{\beta' \mapsto \beta}\Mbar_{g,n}(X, \beta').
  \end{equation}
  Since degree is constant in (connected) families, each stack $\Mbar_{g,n}(X, \beta')$ is open in $\cM(X)$. In particular this is a decomposition into connected components.
Since $X$ is smooth, the morphism \eqref{eq:candidatepot} defines a relative perfect obstruction theory on $\cM(X)$.
By Theorem \ref{thm:mainthm} we see there is a moduli space of $p$-fields $\cM(Y, E)$ containing $\Mbar_{g,n}(X, \beta)$ as a closed substack, and a class $[\cM(Y, E)]^{\vir}_{\loc}$ satisfying
\begin{equation}\label{eq:stablemaps1}
  [\cM(Y, E)]^{\vir}_{\loc} = (-1)^{\chi(E,\cM(X))}[\Mbar_{g,n}(X, \beta)]^{\vir}\quad \text{in}\;A_*(\Mbar_{g,n}(X, \beta)).
\end{equation}
Using the decomposition \eqref{eq:stablemaps2} we rewrite the right hand side of \eqref{eq:stablemaps1} as
\begin{equation*}
  (-1)^{\chi(E,\cM(X))}[\Mbar_{g,n}(X, \beta)]^{\vir}
  = \sum_{\beta' \mapsto \beta} (-1)^{\chi(E,\Mbar_{g,n}(X, \beta'))}[\Mbar_{g,n}(X, \beta')]^{\vir},
\end{equation*}
which completes the proof.
\end{proof}

\subsection{Application to quasimaps}\label{sec:quasimaps}

Let $Y=[W/G]$ where $W$ is an affine l.c.i. variety and $G$ is a
reductive group acting on $W$.
Choose a character $\theta$ of $G$ such that
$W^s_{\theta} = W^{ss}_{\theta}$ is smooth and nonempty and has finite
$G$-stabilizers.
Let $E$ be a $G$-equivariant vector bundle on $W$ with a
$G$-equivariant regular section $s$ whose zero locus $U$ has smooth
intersection with $W^s_{\theta}$.
Fix nonnegative integers $g, n$ and a positive rational number
$\epsilon$, and choose a class $\beta \in \Hom(\Pic^G(W), \QQ)$. 

\begin{proof}[Proof of Corollary \ref{cor:quasimaps}]
Observe first that the $G$-equivariant sheaf $\E$ descends to a sheaf $\overline \E$ on $[W/G]$ and that $s$ descends to a regular section $\overline s$ with zero locus $X = [U/G]$. Moreover, $U$ is an affine l.c.i. variety with $G$-action. It is straightforward to check that
\[
(W^s_\theta \cap U) \subset U^s_{\theta} \subset U^{ss}_{\theta} \subset (W^{ss}_\theta \cap U),
\]
so $W^{ss}_{\theta} = W^s_{\theta}$ implies that $U^{ss}_{\theta} = U^s_{\theta} = U \cap W^{ss}_{\theta}$, and by assumption this locus is smooth. Hence we may consider moduli of $\epsilon$-stable quasimaps to $U\sslash_{\theta} G$.

Let $\iota\colon U \rightarrow W$ denote the inclusion, and let $\cM(Y)$ be the moduli space of $\epsilon$-stable quasimaps $\Mbar_{g,n}^\epsilon(W\sslash_\theta G, \beta)$. We note that the definition in \cite[Def~2.1]{CCK15} guarantees the existence of an ample line bundle as required in Section \ref{sec:introdefs} assumption (4). By Lemma \ref{lem:quasimapcompare} this is (isomorphic to) an open substack of $\Sec{\Ct \times [W/G]}{\Ct}$ and \eqref{eq:candidatepot} is a relative perfect obstruction theory defining the same virtual cycle as the one defined in \cite[Sec~2.4.5]{CCK15}. As in the proof of Corollary \ref{cor:stablemaps}, the moduli space $\cM(X)$ is a disjoint union over $\beta' \in \Hom(\Pic^G(U), \QQ)$ with $\iota_*\beta' = \beta$. An argument analogous to 
the one used in Section \ref{sec:stablemaps} completes the proof of the corollary.
\end{proof}

\subsection{Relation to the original construction of Chang--Li}
We compare our construction to that in \cite{CL12}, that is, to their moduli space $\overline{\cM}_g(\PP^4, d)^p$ and its relative perfect obstruction theory (defined in \cite[3.1]{CL12} and \cite[Prop~3.1]{CL12}, respectively). The easiest comparison is to choose $Y=[\CC^5/\CC^*]$ and $E$ the line bundle determined by the fifth power of the regular representation, and set $\cM(Y) = \overline{\cM}_{g}(\PP^4, d)$. Observe that $\cM(Y)$ is an open substack of $\Sec{\Ct \times \PP^4}{\Ct}$, which is in turn an open substack of $\Sec{\Ct \times [\CC^5/\CC^*]}{\Ct}$ via the embedding $\PP^4 \subset [\CC^5/\CC^*]$. As before, set $Z = Vb_{\Ct\times Y}(\omega \otimes \E^\vee).$

Let $\C_{\cM(Y)} \rightarrow \cM(Y)$ be the universal curve and $f\colon \C_{\cM(Y)} \rightarrow [\CC^5/\CC^*]$ the universal map. The moduli space $\overline{\cM}_g(\PP^4, d)^p$ is defined to be $\Sec{f^*E \otimes \omega}{\C_{\cM(Y)}}$. By Lemmas \ref{lem:basechange} and \ref{lem:basechange2} this is canonically isomorphic to our moduli space
\[
\cM(Y, E) = \cM(Y) \times_{\Sec{\Ct \times [\CC^5/\CC^*]}{\Ct}} \Sec{Z}{\Ct}.
\]

To compare the perfect obstruction theories, we use the following diagram.
\begin{equation}
\begin{tikzcd}
&&Z\arrow[d] \\
& Vb(\mathcal{L}^{\oplus 5}\oplus \mathcal{P}) \arrow[d] \arrow[ur] & \Ct\times B\CC^* \arrow[d]&\\
\C_{\cM^p} \arrow[ur]  \arrow[r] \arrow[d]& \C_{\mathfrak{D}_g} \arrow[r] \arrow[d] \arrow[ur]& \Ct \arrow[d] &\\
\cM^p \arrow[r] &\mathfrak{D}_g \arrow[r] & \mathfrak{M}_g&\\
\end{tikzcd}
\end{equation}
Here, $\mathfrak{D}_g$ is the Picard stack of $\mathfrak{M}_g$, which is identified with $\Sec{\Ct \times B\CC^*}{\Ct}$, and $\mathscr{L}$ is its universal line bundle, with $\mathscr{P} = \mathscr{L}^{-\otimes 5}\otimes \omega$. The stack $\cM^p$ is $\cM(Y, E)=\overline{\cM}_g(\PP^4, d)^p.$ All quadrilaterals in this diagram are fibered. The obstruction theory of \cite[Prop~3.1]{CL12} is the canonical one \eqref{eq:candidatepot} on $\cM^p$, relative to $\mathfrak{D}_g$; our obstruction theory is the canonical one \eqref{eq:candidatepot} relative to $\mathfrak{M}_g$. These are related by a morphism of distinguished triangles, and in particular induce the same virtual cycle, by an argument analogous to that of Lemma \ref{lem:quasimapcompare}.

\subsection{Quantum Lefschetz}
\label{sec:lefschetz}
We will prove Theorem~\ref{thm:lefschetz} in this section.
Recall that we assume
\begin{equation*}
  H^1(C, f^* \E) = 0
\end{equation*}
for every closed point $[f\colon C \to Y] \in \cM(Y)$.
Hence the zero section $\cM(Y) \to \cM(Y, E)$ is an isomorphism and $R^0 \pi_* f^* \E$ is a locally free
sheaf.
  
Pushing forward \eqref{eq:thmb} under the inclusion
$\iota\colon \cM(X) \to \cM(Y, E)$, gives
\begin{equation}
  \label{eq:lefschetz1}
  \iota_* [\cM(X)]^\vir = (-1)^{\chi(E, \cM(Y))} [\cM(Y, E)]^\vir \qquad \text{in }A_*(\cM(Y)).
\end{equation}
Now $[\cM(Y, E)]^\vir$ and $[\cM(Y)]^\vir$ are two virtual cycles on
the same space but defined via different obstruction theories.

By Lemma~\ref{lem:functoriality1}, there exists a morphism of
distinguished triangles
\begin{equation*}
  \begin{tikzcd}
    \EE_{\cM(Y, E)/\Mgnt} \arrow[r] \arrow[d] & \EE_{\cM(Y)/\Mgnt} \arrow[r] \arrow[d] & R \pi_* (Lf^* \LL_{Z/\Ct \times Y}\otimes \omegabul) \arrow[r] \arrow[d] & {}\\
    \LL_{\cM(Y, E)/\Mgnt} \arrow[r] & \LL_{\cM(Y)/\Mgnt} \arrow[r] & \LL_{\cM(Y, E)/\cM(Y)} \arrow[r] &{},
  \end{tikzcd}
\end{equation*}
so that we have a compatible triple of obstruction theories in the
sense of \cite[Definition~4.5]{Ma12}.
Hence, by \cite[Theorem~4.8, Example~3.17]{Ma12}, we have
\begin{equation}
  \label{eq:lefschetz2}
  [\cM(Y, E)]^\vir = \pi^! [\cM(Y)]^\vir,
\end{equation}
where $\pi^!$ is virtual pullback via the projection
$\cM(Y, E) \to \cM(Y)$.
Note that
\begin{equation*}
  (R \pi_* Lf^* \LL_{Z/\Ct \times Y}\otimes \omegabul)^\vee
  = R \pi_*(\omega \otimes f^* \E^\vee)
  = R \pi_*(f^* \E)^\vee
  = (R^0 \pi_* f^* \E)^\vee[1].  
\end{equation*}
Going through the definition of virtual pullback
(\cite[Construction~3.6]{Ma12}), we see that
\begin{equation}
  \label{eq:lefschetz3}
  \pi^! [\cM(Y)]^\vir
  = (-1)^{\chi(E, \cM(Y))} e(R^0 \pi_* f^* \E) \cap [\cM(Y)]^\vir.
\end{equation}
Combining \eqref{eq:lefschetz1}, \eqref{eq:lefschetz2} and
\eqref{eq:lefschetz3}, we conclude the proof of
Theorem~\ref{thm:lefschetz}.

\appendix

\section{Summary of results about the moduli of sections}
\label{sec:moduli-sections}

In this appendix we collect some results about the moduli of sections defined in \eqref{eq:defsec} and its candidate obstruction theory defined in \eqref{eq:candidatepot}. Most, if not all, of these results are well-known, but we could not find references for the proofs. Since our argument relies heavily on these properties, we give a coherent treatment here.

Throughout this appendix, all algebraic stacks a quasi-separated and locally finite type over $\CC$. We fix such an algebraic stack $\mathfrak{U}$ and
$\pi\colon \C\rightarrow \mathfrak{U}$ is a flat finitely-presented
family of connected, twisted (nodal) curves in the sense of \cite{AV02}. By \cite[Prop~2.2.6]{webb-thesis}, such a family is equipped with a functorial pair $(\omegabul_{\fU}, tr_{\fU})$ where the complex $\omegabul_{\fU}$ is represented by $\omega_{\fU}[1]$, with $\omega_{\fU}$ an invertible sheaf in the lisse-\'etale site of $\C$; and $tr_{\fU}: R\pi_*\omegabul_{\fU} \rightarrow \OO_{\fU}$ is a morphism in the derived category.

\subsection{Properties of the moduli}
Suppose we have morphisms of algebraic stacks 
\begin{equation}\label{eq:tower}
Z \rightarrow W\rightarrow \C\xrightarrow{\pi} \mathfrak{U}
\end{equation}
where both $Z\rightarrow \mathfrak{U}$ and $W\rightarrow \mathfrak{U}$ have affine stabilizers. We prove some canonical isomorphisms of moduli of sections. The following observation will be useful.

\begin{lemma}\label{lem:fibered_sections}
Fix a diagram of algebraic stacks with morphisms as below, so that the square is fibered (and includes a 2-morphism $\alpha$):
\[
\begin{tikzcd}
& F \arrow[r, "z"] \arrow[d, "y"] & D \arrow[d] \\
A  \arrow[r, "x"] & B \arrow[r] & C
\end{tikzcd}
\]
Then the arrow $F \rightarrow D$ induces an equivalence of groupoids
\[
\Hom_B(A, F) \xrightarrow{\sim} \Hom_C(A, D).
\]
\end{lemma}
\begin{proof}
Straightforward; see, for example, \cite[Lem~2.3.1]{webb-thesis}.
\end{proof}

In the context of \eqref{eq:tower}, on $\Sec{W}{\C}$ we have the universal curve (pullback of $\C$) and universal section, denoted $f\colon \C_{\Sec{W}{\C}} \rightarrow W$.

\begin{lemma}\label{lem:basechange}
  Let $f^*Z$ denote the fiber product $\C_{\Sec{W}{\C}}\times_W Z$. Then there is a canonical isomorphism
\[ \Sec{f^*Z}{\C_{\Sec{W}{\C}}} \cong \Sec{Z}{\C}\]
of stacks over $\mathfrak{U}$.
\end{lemma}
\begin{proof} 
The canonical morphism $\Phi\colon \Sec{f^*Z}{\C_{\Sec{W}{\C}}}\rightarrow \Sec{Z}{\C}$ is a morphism of categories fibered in groupoids over $\mathfrak{U}$, so to show $\Phi$ is an equivalence, it suffices to study the induced map on fibers over a scheme $T \rightarrow \mathfrak{U}$.

We compute the fiber of $\mathcal{F}:= \Sec{f^*Z}{\C_{\Sec{W}{\C}}}$.
The fiber of $\mathcal{F}$ over an arrow $T\rightarrow \Sec{W}{\C}$ is $\Hom_{\C_{\Sec{W}{\C}}}(\C_T, f^*Z)$; by Lemma \ref{lem:fibered_sections} this is equivalent to $\Hom_W(\C_T, Z)$. Hence $\mathcal{F}(T \rightarrow \mathfrak{U})$ is the groupoid of dotted arrows
\[
\begin{tikzcd}
& Z\arrow[d, "q"]\\
& W \arrow[d, "p"] \\
\C_T \arrow[d] \arrow[ur, dashrightarrow, ""{name=A, above}, ""{name=C, below}] \arrow[uur, dashrightarrow, ""{name=B, below}] \arrow[r, ""{name=D, above}, "i"'] \arrow[Rightarrow, from=B, to=A]\arrow[Rightarrow, from=C, to=D]& \C\arrow[d]\\
T \arrow[r] & \mathfrak{U}
\end{tikzcd}
\]
Specifically, an object of $\mathcal{F(T)}$ is a tuple $(z, w, \tau, \omega)$ where $z\colon \C_T \rightarrow Z$ and $w\colon \C_T \rightarrow W$ are 1-morphisms, and $\tau\colon q\circ z \rightarrow w$ and $\omega\colon p \circ w \rightarrow i$ are 2-morphisms. An arrow in $\mathcal{F}(T)$ from $(z_1, w_1, \tau_1, \omega_1)$ to $(z_2, w_2, \tau_2, \omega_2)$ is a pair of 2-morphisms $\alpha\colon w_1 \rightarrow w_2$ and $\beta\colon z_1 \rightarrow z_2$ such that $\omega_1 = \omega_2 \circ p(\alpha)$ and $\alpha \circ \tau_1 = \tau_2 \circ q(\beta)$.

Now let $\mathcal{G}$ be the usual construction of the fiber product for the diagram\footnote{Compare with \cite[Tag~06N7]{stacks-project}.}
\[
\begin{tikzcd}
\mathcal{G} \arrow[r, "\pi_2"] \arrow[d] &\Sec{Z}{\C}\arrow[d]\\
\Sec{W}{\C} \arrow[r, "id"] &\Sec{W}{\C}
\end{tikzcd}
\]
The map $\Phi$ factors as $\mathcal{F} \xrightarrow{\Phi'} \mathcal{G} \xrightarrow{\pi_2} \Sec{Z}{\C}.$
Of course $\pi_2$ is an equivalence of stacks over $\mathfrak{U}$. On the other hand, we claim that $\Phi'$ induces the literal identity map from $\mathcal{F}(T)$ to $\mathcal{G}(T)$. By definition, an object of the fiber of $F$ is a tuple $(w, \omega; z, \zeta; \tau)$, where $w\colon \C_T \rightarrow W$ and $z\colon \C_T \rightarrow Z$ are 1-morphisms, $\omega\colon p\circ w \rightarrow i$ and $\zeta\colon p\circ q \circ z \rightarrow i$ are 2-morphisms, and $\tau\colon q\circ z \rightarrow w$ is a 2-morphism such that $\zeta = \omega \circ p(\tau)$. The final condition determines $\zeta$ from the other data, and hence these objects are literally the same as the objects of $\mathcal{F}$. Arrows in these two groupoids are also literally the same.

\end{proof}

\begin{lemma}\label{lem:basechange2}
Let $Z\rightarrow \C \rightarrow \mathfrak{U}$ be as above. Suppose $Z' \rightarrow \C' \rightarrow \mathfrak{U}'$ is another tower of the same type, and suppose we have a commuting diagram of fibered squares
\[
\begin{tikzcd}
Z' \arrow[r] \arrow[d]& Z \arrow[d] \\
\C' \arrow[r] \arrow[d] & \C \arrow[d] \\
\mathfrak{U}' \arrow[r, "f"] & \mathfrak{U}
\end{tikzcd}
\]
Then there is a canonical isomorphism
\begin{equation}\label{eq:basechange2}
\Sec{Z'}{\C'} \cong \Sec{Z}{\C}\times_\mathfrak{U} \mathfrak{U}'
\end{equation}
of stacks over $\mathfrak{U}'.$
\end{lemma}
\begin{proof}
Let $\mathcal{F} = \Sec{Z}{\C}\times_{\mathfrak{U}}\mathfrak{U}'$. First, observe that a slight extension of the argument in \cite[Tag~06N7]{stacks-project} shows that $\mathcal{F}$ is indeed fibered in groupoids over $\mathfrak{U}'$. So to show that the canonical map $\Phi\colon \Sec{Z'}{\C} \rightarrow\mathcal{F}$ is an equivalence, it suffices to show it is an equivalence on the fiber over arbitrary $x\colon T \rightarrow \mathfrak{U}'$. 

The fiber $\mathcal{F}(T)$ has for objects tuples $(a, \alpha, n, \nu)$ where (letting $C_a = C \times_{\mathfrak{U}, a} T$) $a\colon T \rightarrow U$ and $n\colon C_a \rightarrow Z$ are 1-morphisms and $\alpha\colon f \circ x \rightarrow a$ and $\nu$ are 2-morphisms ($\nu$ witnesses the commutativity of a triangle, one of whose sides is $n$). A morphism in $\mathcal{F}(T)$ from $(a, \alpha, n, \nu)$ to $(b, \beta, m, \mu)$ is a tuple $(\tau, \sigma)$ where $\tau\colon a \rightarrow b$ and $\sigma\colon n \rightarrow m \circ c_{\tau}$ are 2-morphisms (here $c_{\tau}\colon C_a\rightarrow C_b$ is the morphism induced by $\tau$), such that (1) $\beta^{-1}\circ \tau \circ \alpha$ is the identity, and (2) the 2-cell with faces $\sigma$, $\nu$, $\mu$, and one other face determined by $\tau$ is commutative. 

The fiber $\Sec{Z'}{\C'}$ is by Lemma \ref{lem:fibered_sections} equal to $\Hom_{\C}(\C_T, Z)$. This groupoid has for objects pairs $(n, \nu)$ where $n\colon \C_T \rightarrow Z$ is a 1-morphism and $\nu$ is a 2-morphism witnessing the commutativity of the triangle over $\C$. A morphism from $(n, \nu)$ to $(m, \mu)$ is a 2-morphism $\sigma\colon n \rightarrow m$ such that the 2-cell with $\sigma, \nu,$ and $\mu$ commutes.

Let $\Phi_T\colon \Sec{Z'}{\C'} \rightarrow \mathcal{F}(T)$ be the restriction of $\Phi$ to the fiber. Then $\Phi_T$ sends $(n, \nu)$ to $(f \circ x, id, n, \nu)$ and $\sigma$ to $(id, \sigma)$. The map $\Phi_T$ is essentially surjective because $\alpha$ induces an isomorphism from an object in the image of $\Phi_T$ to $(a, \alpha, n, \nu)$. It is fully faithful because if $\beta=\alpha=id$, then condition (2) forces $\tau = id$.
\end{proof}

\begin{lemma}
  \label{lem:closed}
  There is a natural morphism $\Sec{Z}{\C} \rightarrow \Sec{W}{\C}$. If $Z \rightarrow W$ is a closed embedding, then so is $\Sec{Z}{\C} \rightarrow \Sec{W}{\C}$.
\end{lemma}
\begin{proof}
  Let $\mathfrak{S}' = \Sec{Z}{\C}$ and $\mathfrak{S} = \Sec{W}{\C}$.
  Since $\mathfrak{S}' \rightarrow \mathfrak{U}$ is already locally of finite
  type, by \cite[Tag~04XV]{stacks-project} it suffices to show that
  $\iota\colon \mathfrak{S}' \rightarrow \mathfrak{S}$ is universally closed and a
  monomorphism.
  The monomorphism property is immediate using the characterization in
  \cite[Tag~04ZZ]{stacks-project}.
  By Lemma \ref{lem:basechange2}, to show that $\iota$ is universally closed it suffices to prove that
  it is a closed map, and since we already know that $\iota$ is a
  momomorphism, it suffices to show that $\iota(\mathfrak{S}')$ is closed in $\mathfrak{S}$.
  Now the set $\iota(\mathfrak{S}')$ consists of points whose $\pi_{\mathfrak{S}}$-fibers map
  completely into $Z$.
  Hence, $\mathfrak{S} \setminus \iota(\mathfrak{S}') = \pi_{\mathfrak{S}}(\n^{-1}(W \setminus Z))$.
  Since $\pi_{\mathfrak{S}}$ is flat, and hence open, this implies that $\iota(\mathfrak{S}')$
  is closed, which finishes the proof of the lemma.
  \end{proof}

\subsection{Properties of the obstruction theory}
\subsubsection{An adjunction-like morphism}
Let $\D({\C})$ (resp. $\D({\fU})$) denote the unbounded derived category of sheaves of $\OO$-modules on $\C$ (resp. $\fU$) in the lisse-\'etale topology. Let $\Dqc(\C)$ and $\Dqc(\fU)$ denote the corresponding subcategories on objects with quasi-coherent cohomology. Define an adjunction-like morphism
\[
a: \Hom_{\D(\C)}(F, \pi^*G) \rightarrow \Hom_{\D({\fU})}(R\pi_*(F\otimes \omegabul), G)
\]
by sending $f: F \rightarrow \pi^*G$ to the composition
\[
R\pi_*(F\otimes \omegabul) \xrightarrow{R\pi_*(f\otimes id)} R\pi_*(\pi^*G\otimes \omegabul) \xleftarrow{\sim} G \otimes R\pi_*\omegabul \xrightarrow{tr} G
\]
where the isomorphism is the projection formula and $tr$ is the trace map. Observe that $a$ is functorial in both arguments, meaning
\begin{enumerate}
\item Given $F' \in \D({\C})$ and $g \in \Hom_{\D({\C})}(F', F)$, we have $a(f \circ g) = a(f) \circ R\pi_*(g \otimes id)$.
\item Given $G' \in D({\fU})$ and $g \in \Hom_{\D({\fU})}(G, G')$, we have $a(\pi^*g \circ f) = g \circ a(f).$
\end{enumerate}
The next lemma, inspired by \cite[Lem~4.1]{ACGS-puncture}, says that $a$ commutes with pullback.

\begin{lemma}\label{lem:technical}
The adjunction-like map $a$ commutes with arbitrary basechange. Precisely, given a fiber square
\[
\begin{tikzcd}
K \arrow[r, "\mu_K"] \arrow[d, "\pi"] & \C \arrow[d, "\pi"]\\
B \arrow[r, "\mu_B"] & \fU
\end{tikzcd}
\]
and a morphism $f: F\rightarrow \pi^*G$, we have  $\mu_K^*f: \mu_K^*X \rightarrow \mu_K^*\pi^*Y = \pi^*\mu_B^*Y$, and the following diagram commutes:
\[
\begin{tikzcd}
\mu_B^*R\pi_*(F\otimes \omegabul) \arrow[r, "\mu_B^*a(f)"] \arrow[d, "\fa_\mu"', "\sim"] & \mu_B^*G \\
R\pi_*(\mu_K^*F\otimes \omegabul) \arrow[ur, "a(\mu_K^*f)"']
\end{tikzcd}
\]
The isomorphism $\fa_{\mu}$ is functorial in $X$. Moreover, if $\nu_B: B' \rightarrow B$ is a morphism of algebraic stacks with $K' := K\times_B B'$ and $\nu_K: K' \rightarrow K$ the projection, then these isomorphisms satisfy the cocycle condition $\fa_{\mu\circ \nu} = \fa_{\nu} \circ \nu_K^* \fa_{\mu}$.
\end{lemma}
Before proving the lemma we note several canonical isomorphisms: 
\begin{align}
&\mu_B^*R\pi_*F \rightarrow R\pi_*\mu_K^*F & F \in \Dqc(\C)\label{eq:basechange-iso}\\
&G \otimes R\pi_*F \rightarrow R\pi_*(\pi^*G \otimes F) & F \in \Dqc(\C), G \in \Dqc(\fU)\label{eq:projection-iso}\\
&\mu_K^*\omegabul_\fU \rightarrow \omegabul_B \label{eq:omega-iso}
\end{align}
The first is \cite[Cor~4.13]{HR17}, the second is \cite[Cor~4.12]{HR17}, and the last is \cite[Prop~2.2.6]{webb-thesis}.

\begin{proof}[Proof of Lemma \ref{lem:technical}.]
 The desired commuting triangle is equivalent to 
\[
\begin{tikzcd}[row sep=small]
\mu_B^*R\pi_*(F\otimes \omegabul) \arrow[rr, "\mu_B^*a(f)"] \arrow[d, "\fa_{\mu}"] && \mu_B^*G \\
R\pi_*(\mu_K^*(F) \otimes \omegabul) \arrow[dr, "R\pi_*(\mu_K^*(f)\otimes \omegabul)"] && \mu_B^*G\otimes R\pi_*\omegabul \arrow[dl, "\pi", "\sim"'] \arrow[u, "tr"] \\ &R\pi_*(\mu_K^*(\pi^*G)\otimes \omegabul) = R\pi_*(\pi^*\mu_B^*G\otimes \omegabul) & 
\end{tikzcd}
\]
We demonstrate this diagram as the composition of three. From left to right, the first is
\begin{equation}\label{eq:tec1}
\begin{tikzcd}
\mu_B^*R\pi_*(F\otimes \omegabul) \arrow[r, "\mu_B^*R\pi_*(f \otimes \omegabul)"] \arrow[d, "\fa_{\mu}"', "\sim"] &[6em] \mu_B^*R\pi_*(\pi^*G \otimes \omegabul) \arrow[d, "\sim"]\\
R\pi_*(\mu_K^*(F)\otimes \omegabul) \arrow[r, "R\pi_*(\mu_K^*(f)\otimes \omegabul)"] & R\pi_*(\mu_K^*\pi^*G\otimes \omegabul)
\end{tikzcd}
\end{equation}
where the vertical arrows are \eqref{eq:basechange-iso} followed by the strong monoidal map (commutativity of $\otimes$ and $\mu_K^*$), and finally \eqref{eq:omega-iso}. It commutes by functoriality of \eqref{eq:basechange-iso} and the strong monoidal map, and because \eqref{eq:omega-iso} and $f$ act on different factors of the tensor product. The second diagram is
\begin{equation}\label{eq:tec2}
\begin{tikzcd}
\mu_B^*R\pi_*(\pi^*G\otimes \omegabul) \arrow[d, "\text{\eqref{eq:basechange-iso}}"] & \mu_B^*(G\otimes R\pi_*\omegabul) \arrow[l, "\mu_B^*\eqref{eq:projection-iso}", "\sim"'] \arrow[d, equal] \\
R\pi_*\mu_K^*(\pi^*G\otimes \omegabul) \arrow[d, equal] &\mu_B^*G\otimes \mu_B^*R\pi_*\omegabul \arrow[d, "\text{\eqref{eq:basechange-iso}}"] \\
R\pi_*(\mu_K^*f^*G\otimes \mu_K^*\omegabul)=R\pi_*(\pi^*\mu_B^*G\otimes \mu_K^*\omegabul)\arrow[d] & \mu_B^*G\otimes R\pi_*\mu_K^*\omegabul \arrow[l, "\eqref{eq:projection-iso}", "\sim"'] \arrow[d, "id \otimes R\pi_*\text{\eqref{eq:omega-iso}}"]\\
R\pi_*(\mu_K^*f^*G\otimes \omegabul) = R\pi_*(\pi^*\mu_B^*G\otimes \omegabul) & \mu_B^*Y\otimes R\pi_*S \arrow[l, "\text{\eqref{eq:projection-iso}}","\sim"']
\end{tikzcd}
\end{equation}
The top cell commutes by \cite[Lem~A.7(3)]{hall-GAGA} and the bottom cell is functoriality of \eqref{eq:projection-iso}. The final diagram is
\begin{equation}\label{eq:tec3}
\begin{tikzcd}[row sep=small]
\mu_B^*(G\otimes R\pi_*\omegabul) \arrow[d, equal] \arrow[r, "\mu_B^*(id \otimes tr)"] & [6em] \mu_B^*G\arrow[d, equal]\\
\mu_B^*Y\otimes \mu_B^*R\pi_*\omegabul \arrow[r, "id \otimes \mu_B^*tr"] \arrow[d, "\text{\eqref{eq:basechange-iso}}"'] & \mu_B^*G \\
\mu_B^*G\otimes R\pi_*\mu_K^*\omegabul \arrow[r, "id \otimes R\pi_*\alpha"']&
\mu_B^*G\otimes R\pi_*\omegabul \arrow[u, "id \otimes tr"']
\end{tikzcd}
\end{equation}
where the bottom cell is the diagram \cite[Prop~2.2.6]{webb-thesis}.
One sees from the definition of $a$ that the composition of the top arrows in \eqref{eq:tec1}, \eqref{eq:tec2}, and \eqref{eq:tec3} is $\mu_B^*a(f).$ Finally, $\fa_{\mu}$ is functorial in $F$ by its definition, and it satisfies the cocycle condition because each of the morphisms in its definition do: for \eqref{eq:basechange-iso} the cocycle condition is \cite[Tag~0E47]{stacks-project} and for \eqref{eq:omega-iso} it is \cite[Prop~2.2.6]{webb-thesis}.
\end{proof}



\subsubsection{Functoriality}
Suppose we have a commuting diagram of algebraic stacks
\begin{equation}\label{eq:ot-diagram}
\begin{tikzcd}
& Z \arrow[r] & W \arrow[d] \\
K_{Z} \arrow[d, "\pi"]\arrow[ur, "f"] \arrow[r] & K_{W} \arrow[r] \arrow[d] \arrow[ur] & \C \arrow[d] \\
B_{Z} \arrow[r] & B_{W} \arrow[r] & \fU
\end{tikzcd}
\end{equation}
where the squares in the bottom row are fibered. To a diagram of the form \eqref{eq:ot-diagram} we associate a morphism in $\Dqc(B_{Z})$ as follows. We have a morphism in $\Dqc(K_{Z})$ consisting of canonical morphisms of cotangent complexes:
\begin{equation}\label{eq:build1}
Lf^*\LL_{Z/W} \rightarrow \LL_{K_{Z}/K_{W}} \xleftarrow{\sim} \pi^*\LL_{B_{Z}/B_{W}}.
\end{equation}
We may apply the adjunction-like morphism $a$ to \eqref{eq:build1}, obtaining
\begin{equation}\label{eq:generalpot}
\phi_{B_{Z}/B_{W}} : \EE_{B_{Z}/B_{W}} \rightarrow \LL_{B_{Z}/B_{W}}, \quad \quad \quad \EE_{B_{Z}/B_{W}}:= R\pi_*(Lf^*\LL_{Z/W} \otimes \omegabul_{B_Z}).
\end{equation}
Observe that when $K_W = W = \C$, $B_W=\fU$ and $B_Z = \Sec{Z}{\C}$ we recover \eqref{eq:candidatepot}.
\begin{remark}\label{rmk:equiv}
It follows from the functoriality of $a$ in the first argument that when either $W \rightarrow \C$ or $K_W \rightarrow \C$ is flat, the morphism $\phi_{B_Z/B_W}$ defined by the diagram \eqref{eq:ot-diagram} is isomorphic to $\phi_{B_{Z'}/B_W}$ defined by the diagram obtained from \eqref{eq:ot-diagram} by setting $Z' = Z\times_s K_W$, $B_W = \fU$, and $K_W = W = \C$.
\end{remark}

The morphism \eqref{eq:generalpot} inherits the functoriality properties of the morphisms of cotangent complexes used to define it.

\begin{lemma}\label{lem:functoriality1}
Let $Z \xrightarrow{g} W \rightarrow V$ be morphisms of algebraic stacks over $\C$. Let $B_{Z}, B_{W}$, and $B_{V}$ be quasi-separated algebraic stacks locally of finite type fitting in a fiber diagram
\[
\begin{tikzcd}
K_{Z} \arrow[d] \arrow[r] & K_{W} \arrow[r] \arrow[d] & K_{V} \arrow[r]\arrow[d] &\C \arrow[d] \\
B_{Z} \arrow[r, "\mu"] & B_{W} \arrow[r] & B_{V} \arrow[r]& \fU
\end{tikzcd}
\]
and suppose we are given morphisms $f: K_{Z} \rightarrow Z$, $ K_{W}\rightarrow W$, and $K_{V}\rightarrow V$ so that the analog of \eqref{eq:ot-diagram} commutes. Then there is a morphism of (canonical) distinguished triangles
\begin{equation}\label{eq:a.2.2-1}
\begin{tikzcd}
L\mu^*\EE_{B_W/B_V} \arrow[r] \arrow[d, "L\mu^*\phi_{B_W/B_V}"]&\EE_{B_Z/B_V} \arrow[r]\arrow[d, "\phi_{B_Z/B_V}"]& \EE_{B_Z/B_W} \arrow[r] \arrow[d, "\phi_{B_Z/B_W}"]& {}\\
L\mu^*\LL_{B_W/B_V} \arrow[r] & \LL_{B_Z/B_V} \arrow[r] & \LL_{B_Z/B_W} \arrow[r] & {}
\end{tikzcd}
\end{equation}\end{lemma}

\begin{proof}By the functoriality in \cite[Lem~2.2.12]{webb-thesis} we have a morphism of distinguished triangles
\begin{equation}\label{eq:a.2.2}
\begin{tikzcd}
Lf^*Lg^*\LL_{W/V} \arrow[r] \arrow[d]&Lf^*\LL_{Z/V} \arrow[r]\arrow[d]& Lf^*\LL_{Z/W} \arrow[r] \arrow[d]& {}\\
\pi^*L\mu^*\LL_{B_W/B_V} \arrow[r] & \pi^*\LL_{B_Z/B_V} \arrow[r] & \pi^*\LL_{B_Z/B_W} \arrow[r] & {}
\end{tikzcd}
\end{equation}
where each vertical arrow is the compositions of two canonical morphisms of cotangent complexes, one of which is inverted in the derived category. Now apply the adjunction-like operation $a$ to each vertical arrow. Since $a$ is functorial in both arguments and commutes with pullback by Lemma \ref{lem:technical}, we obtain \eqref{eq:a.2.2-1}.

\end{proof}

\begin{lemma}\label{lem:functoriality2}
Suppose we have commuting squares
\begin{equation}\label{eq:square1}
\begin{tikzcd}
\arrow[d]W \arrow[r, "g"] & Y \arrow[d] &&& B_W \arrow[r, "\mu_B"] \arrow[d] & B_Y \arrow[d]\\
X \arrow[r] & Z &&& B_X \arrow[r]& B_Z
\end{tikzcd}
\end{equation}
where the square on the left consists of stacks over $\C$ and the square on the right consists of stacks over $\fU$. Let $K_W = \C \times_{\fU} {B_W}$ (similarly for $X, Y$, and $Z$) and suppose we have a map $f_W: K_W \rightarrow W$ (and similarly for $X$, $Y$, and $Z$) such that the resulting large diagram is commutative. Then there is a commuting diagram
\begin{equation}\label{eq:functoriality2}
\begin{tikzcd}
L\mu_B^*\EE_{B_Y/B_Z} \arrow[r] \arrow[d, "L\mu_B^*\phi_{B_Y/B_Z}"] & \EE_{B_W/B_X}\arrow[d, "\phi_{B_W/B_X}"]\\
L\mu_B^*\LL_{B_Y/B_Z} \arrow[r] & \LL_{B_W/B_X}
\end{tikzcd}
\end{equation}
where the top (resp. bottom) horizontal arrow is an isomorphism if the left (resp. right) square in \eqref{eq:square1} is fibered and either $X \rightarrow Z$ or $Y \rightarrow Z$ (resp. $B_X \rightarrow \B_Z$ or $B_Y \rightarrow B_Z$) is flat.
\end{lemma}
\begin{proof}
As in the proof of Lemma \ref{lem:functoriality1}, apply $a$ to the following commuting diagram of canonical morphisms of cotangent complexes.
\[
\begin{tikzcd}
Lf_W^*Lg^*\LL_{Y/Z} \arrow[r] \arrow[d] & Lf_W^*\LL_{W/X} \arrow[d] \\
\pi^*L\mu^*_B \LL_{B_Y/B_Z} \arrow[r] & \pi^*\LL_{B_W/B_X}
\end{tikzcd}
\]
\end{proof}
\subsubsection{A quasimap example}
Fix a complex affine reductive group $G$.
Let $\mathfrak{M} = \mathfrak{M}^{\mathrm{tw}}_{g,n}$ denote the moduli space of prestable orbifold curves of genus $g$ with $n$ markings, and let $\mathfrak{\Ct}$ be its universal curve. Denote by $\mathfrak{B} = \Sec{\Ct\times BG}{\Ct}$ the moduli stack of prestable orbicurves together with a principal $G$-bundle, and let $\mathfrak{P} \rightarrow \Ct_\mathfrak{B}$ be the universal principal bundle over its universal curve.

Now let $Y$ be an affine l.c.i. variety with an action by $G$. Let $\theta$ be a character of $G$ with $Y^{ss}_\theta = Y^s_\theta$ smooth. Fix $\epsilon >0$ and $\beta \in \Hom(Pic([Y/G]), \QQ)$. Then \cite{CCK15} defines a moduli space of $\epsilon$-stable quasimaps $\Mbar_{g,n}^\epsilon([Y/G], \beta)$ as an open substack of $\Sec{\mathfrak{P}\times_GY}{\Ct_\mathfrak{B}}$. By Lemma \ref{lem:basechange}, the stack $\Mbar_{g,n}^\epsilon([Y/G], \beta)$ is identified with an open substack of $\Sec{\Ct\times[Y/G]}{\Ct}$.

We now investigate the obstruction theory (see also the discussion in \cite[Section 5.3]{CiKi10}). Setting $\Mbar = \Mbar_{g,n}^\epsilon([Y/G], \beta),$ we have the following commuting diagram.
\begin{equation}
\begin{tikzcd}[row sep=9pt]
&&\Ct \times [Y/G]\arrow[d, "q"] \\
& \fP\times_G Y \arrow[d] \arrow[ur] & \Ct\times BG \arrow[d]&\\
\Ct_\Mbar \arrow[ur, "u"]  \arrow[r] \arrow[d]& \Ct \arrow[r] \arrow[d] \arrow[ur]& \Ct \arrow[d] &\\
\Mbar \arrow[r, "\mu"] &\mathfrak{B} \arrow[r] & \mathfrak{M}&\\
\end{tikzcd}
\end{equation}
By Lemma \ref{lem:functoriality1} and Remark \ref{rmk:equiv} we have the following morphism of distinguished triangles:
\[
\begin{tikzcd}
L\mu^*\EE_{\mathfrak{B}/\mathfrak{M}} \arrow[d, "L\mu^*\phi_{\mathfrak{B}/\mathfrak{M}}"] \arrow[r] & \EE_{\Mbar/\mathfrak{M}} \arrow[d, "\phi_{\Mbar/\mathfrak{M}}"] \arrow[r] & \EE_{\Mbar/\mathfrak{B}} \arrow[d, "\phi_{\Mbar/\mathfrak{B}}"] \arrow[r] & {}\\
L\mu^*\LL_{\mathfrak{B}/\mathfrak{M}}  \arrow[r] & \LL_{\Mbar/\mathfrak{M}} \arrow[r] & \LL_{\Mbar/\mathfrak{B}}  \arrow[r] & {}
\end{tikzcd}
\]
where the morphism in the third column is constructed by setting $Z=\fP \times_G Y$, $W=\C=K_W = \Ct$, and $B_W = \fU = \mathfrak{B}$ in \eqref{eq:ot-diagram}. The map $\mathfrak{B} \rightarrow \mathfrak{M}$ is ``relatively Artin'' so that the complex $L\mu^*\EE_{\mathfrak{B}/\mathfrak{M}}$ is perfect in $[-1,1]$. It is known that even in this setting, the morphism $\phi_{\mathfrak{B}/\mathfrak{M}}$ is an obstruction theory, meaning that $h^1$ and $h^0$ of $L\mu^*\phi_{\mathfrak{B}/\mathfrak{M}}$ are isomorphisms and $h^{-1}$ of the same is a surjection (see \cite[Ex~8.12]{AP19} or \cite{webb-thesis}). Then because $h^{i}(L\mu^*\EE_{\mathfrak{B}/\mathfrak{M}})=0$ for $i<0$ it follows that $\mathfrak{B}\rightarrow \mathfrak{M}$ is smooth and $\phi_{\mathfrak{B}/\mathfrak{M}}$ is a quasi-isomorphism.
Then argument of \cite[Prop~3]{KKP03} shows that $\phi_{\Mbar/\Mgnt}$ and $\phi_{\Mbar/\mathfrak{B}}$ induce the same virtual class on $\Mbar$.

We have proved the following lemma.

\begin{lemma}\label{lem:quasimapcompare}
The stack of $\epsilon$-stable quasimaps $\Mbar_{g,n}^\epsilon([Y/G], \beta)$ is canonically isomorphic to an open substack of $\Sec{\Ct\times[Y/G]}{\Ct}$. Moreover, the restriction of \eqref{eq:candidatepot} to this substack is a perfect obstruction theory, and it induces the same virtual fundamental class as the perfect obstruction theory of \cite[4.4.1]{CKM} and \cite[Sec~2.4.5]{CCK15}.
\end{lemma}

\subsection{Equivariance}
Let $G$ be a flat, separated group scheme, finitely presented over $\CC$. 
In this section we use the definitions of $G$-stacks
and equivariant morphisms in \cite{Ro05}.
If $Z$ is an algebraic stack with $G$-action, we say that a complex (resp. diagram of complexes) in $\Dqc(Z)$ \textit{is G-equivariant} if it is isomorphic to the pullback of a complex (resp. diagram of complexes) in $\Dqc([Z/G])$. 

\begin{lemma}\label{lem:space-action}
Suppose we are given $G$-equivariant stacks and morphisms $Z \rightarrow W \rightarrow \C \rightarrow \fU$. If $B_Z = \Sec{Z}{\C}$ and $B_W = \Sec{W}{\C}$ in \eqref{eq:ot-diagram}, then there are $G$-actions on all the stacks in \eqref{eq:ot-diagram} such that the entire diagram is equivariant.
\end{lemma}
\begin{proof}
To simplify the notation we take $Z=W$ and $B_Z=B_W$; the proof in the general case is similar.
We have a commuting diagram
\[
\begin{tikzcd}
&G\times Z\arrow[r] \arrow[d]& Z \arrow[d]\\
G\times \C_{\Sec{Z}{\C}} \arrow[r] \arrow[d] \arrow[ru]&G\times \C\arrow[r]\arrow[d]& \C\arrow[d]\\
G\times \Sec{Z}{\C} \arrow[r] &G\times \mathfrak{U} \arrow[r]& \mathfrak{U}
\end{tikzcd}
\]
where the left part of the diagram is just the product of $G$ with \eqref{eq:ot-diagram}, and the right part is the equivariance of the right column of \eqref{eq:ot-diagram}---in particular, the horizontal arrows are all action by $G$. By the universal property of $\Sec{Z}{\C}$, this diagram factors canonically through the original one \eqref{eq:ot-diagram}. This is the desired equivariance.  
\end{proof}

\begin{lemma}\label{lem:equ-space}
Suppose we are given equivariant stacks and morphisms $Z \rightarrow \C \rightarrow \fU$. There is a natural isomorphism 
\[[\Sec{Z}{\C}/G]\simeq \Sec{[Z/G]}{[\C/G]}\]
where the latter moduli space is constructed using the family $[\C/G] \rightarrow [\fU/G]$. 
\end{lemma}
\begin{proof}
We have a fiber diagram
\[
\begin{tikzcd}[row sep=9pt]
Z \arrow[r] \arrow[d] & {[Z/G]} \arrow[d]\\
\C \arrow[r] \arrow[d] & {[\C/G]} \arrow[d]\\
\mathfrak{U} \arrow[r] & {[\mathfrak{U}/G]}
\end{tikzcd}
\]
where the horizontal maps are fppf covers (see \cite[Theorem~4.1]{Ro05}). Hence by Lemma \ref{lem:basechange2}
we have a fiber square as in the following diagram.
\begin{equation}\label{eq:equivariance}
\begin{tikzcd}
&{[\Sec{Z}{\C}/G]}\arrow[d, "F"]\\
\Sec{Z}{\C} \arrow[r, "\rho_{S}"] \arrow[ur]\arrow[d, "q"] & \Sec{[Z/G]}{[\C/G]} \arrow[d, "p"] \\
\mathfrak{U} \arrow[r, "\rho_\mathfrak{U}"] & {[\mathfrak{U}/G]}
\end{tikzcd}
\end{equation}
 The map $\rho_S$ is equivariant and hence factors as depicted. In fact, the outer trapezoid is also fibered, since it is a commuting diagram of $G$-torsors. Again, both horizontal maps and the diagonal map are fppf covers, so by descent the map labeled $F$ is an isomorphism. 
\end{proof}

\begin{lemma}\label{lem:equ-morphs}Suppose we are given a diagram \eqref{eq:ot-diagram} consisting of $G$-equivariant stacks and morphisms. Then the corresponding morphism $\phi_{B_Z/B_W}$ is canonically $G$-equivariant. Moreover, the conclusions of Lemmas \ref{lem:functoriality1} and \ref{lem:functoriality2} hold $G$-equivariantly.
\end{lemma}
\begin{proof}
For the first claim, observe that there are two fibered squares:
\[
\begin{tikzcd}
\arrow[d]Z \arrow[r ] & {[Z/G]} \arrow[d] &&& B_Z \arrow[r, "p"] \arrow[d] & {[B_Z/G]} \arrow[d]\\
W \arrow[r] & {[W/G]} &&& B_W \arrow[r]& {[B_W/G]}
\end{tikzcd}
\]
where the square on the left consists of stacks over $[\C/G]$ and the square on the right consists of stacks over $[\fU/G]$ (that the square on the right is fibered follows from Lemma \ref{lem:basechange2}). The horizontal maps are flat, so Lemma \ref{lem:functoriality2} yields a commuting square
\[
\begin{tikzcd}
p^*\EE_{[B_Z/G]/[B_W/G]} \arrow[r] \arrow[d, "p^*\phi_{[B_Z/G]/[B_W/G]}"] & \EE_{B_Z/B_W}\arrow[d, "\phi_{B_Z/B_W}"]\\
p^*\LL_{[B_Z/G]/[B_W/G]} \arrow[r] & \LL_{B_Z/B_W}
\end{tikzcd}
\]
which is the desired equivariance of $\phi_{B_Z/B_W}$.

We prove the equivariant version of Lemma \ref{lem:functoriality1} (the proof of Lemma \ref{lem:functoriality2} is similar). Let $p: B_Z \rightarrow [B_Z/G]$ be the quotient map. Then $p$ induces a morphism to \eqref{eq:a.2.2} from the following morphism of distinguished triangles:
\[
\begin{tikzcd}
p^*Lf^*Lg^*\LL_{[W/G]/[V/G]} \arrow[d]\arrow[r] & p^*Lf^*\LL_{[Z/G]/[V/G]} \arrow[r] \arrow[d]& p^*Lf^*\LL_{[Z/G]/[W/G]} \arrow[r] \arrow[d]& {}\\
p^*\pi^*\mu_B^*\LL_{[B_W/G]/[B_V/G]} \arrow[r] & p^*\pi^*\LL_{[B_Z/G]/[B_V/G]} \arrow[r] & p^*\pi^*\LL_{[B_Z/G]/[B_W/G]} \arrow[r] & {}
\end{tikzcd}
\]
where we have used $f$ and $g$ for the analogous maps of quotient stacks. In other words, we consider a commuting diagram that includes four distinguished triangles as the edges of a rectangular prism. The rest of the argument is as in the proof of Lemma \ref{lem:functoriality1}, except that to conclude the final diagram commutes we also use the cocycle condition on $\fa_\mu$ proved in Lemma \ref{lem:technical}. 
\end{proof}

Consider a diagram \eqref{eq:ot-diagram} and define $\mu$ to be the map $B_Z\rightarrow B_W$ given there. We define an \textit{absolute} version of $\phi_{B_Z/B_W}$ to be a (shifted) mapping cone fitting in the following diagram, where the arrow labeled ``$F$'' is defined to make its square commutative:
\begin{equation}\label{eq:absolutepot}
\begin{tikzcd}
\EE^{abs}_{B_Z/B_W} \arrow[r] \arrow[d, "\phi^{abs}_{B_Z/B_W}"] & \EE_{B_Z/B_W} \arrow[r, "F"] \arrow[d, "\phi_{B_Z/B_W}"] & L\mu^*\LL_{B_W}[1] \arrow[r]\arrow[d, equal] & {}\\
\LL_{B_W} \arrow[r] & \LL_{B_Z/B_W} \arrow[r] & L\mu^*\LL_{B_W}[1] \arrow[r] & {}
\end{tikzcd}
\end{equation}

\begin{lemma}
  \label{lem:eqabsolutepot}
The morphism $\phi^{abs}_{B_Z/B_W}\colon \EE^{abs}_{B_Z/B_W} \rightarrow \LL_{B_Z}$ is naturally equivariant.
\end{lemma}
\begin{proof}
Define $E$ to be a (shifted) mapping cone fitting into the following morphism of distinguished triangles:
\begin{equation}\label{eq:apot1}
\begin{tikzcd}
E \arrow[d]\arrow[r] & Lf^*\LL_{[Z/G]/[W/G]} \arrow[r] \arrow[d]& \pi^*L\mu^*L_{[B_W/G]/[\bullet/G]}{[1]} \arrow[r] \arrow[d]& {}\\
\pi^*\LL_{[B_Z/G]/[\bullet/G]} \arrow[r] & \pi^*\LL_{[B_Z/G]/[B_W/G]} \arrow[r] & \pi^*L\mu^*L_{[B_W/G]/[\bullet/G]}{[1]} \arrow[r] & {}
\end{tikzcd}
\end{equation}
Here $\bullet=\Spec(\CC)$ with the trivial $G$-action. We likewise have a diagram
\begin{equation}\label{eq:apot2}
\begin{tikzcd}
 & Lf^*\LL_{Z/W} \arrow[r] \arrow[d]& \pi^*L\mu^*L_{B_W}{[1]}  \arrow[d]& {}\\
\pi^*\LL_{B_Z} \arrow[r] & \pi^*\LL_{B_Z/B_W} \arrow[r] & \pi^*L\mu^*L_{B_W}{[1]} \arrow[r] & {}
\end{tikzcd}
\end{equation}
 Let $p: B_Z \rightarrow [B_Z/G]$ denote the qutoient map. Then $p$ induces an \textit{isomorphism} from (the relevant portion of) $p^*$ applied to \eqref{eq:apot1}, to the diagram \eqref{eq:apot2}. We define the missing mapping cone in \eqref{eq:apot2} to be $p^*E$, together with the constructed (iso)morphisms. Now apply the adjunction-like morphism $a$ and conclude as in the proof of Lemmas \ref{lem:functoriality1} and \ref{lem:equ-morphs}.
\end{proof}

Let $T=G$ be a torus and suppose we have a diagram \eqref{eq:ot-diagram} where all stacks and morphisms are $T$-equivariant. Suppose moreover that $B_Z$ and $B_W$ are Deligne-Mumford, that $T$ acts trivially on $B_W$, and we let $F \rightarrow B_Z$ denote the $T$-fixed locus. If $B_W=\Spec(\CC)$ and $\phi_{B_Z/B_W}$ is a perfect obstruction theory, it is shown in \cite{CKL17} that the following composition is a perfect obstruction theory for $F$:
\begin{equation}\label{eq:ckl-pot}
\EE_{B_Z/B_W}|_F^{\fix} \xrightarrow{\phi_{B_Z/B_W}|^{\fix}_F} \LL_{B_Z/B_W}|_F \rightarrow \LL_{F/B_W}.
\end{equation}
On the other hand, we may construct the morphism $\phi_{F/B_W}$ by simply replacing $B_Z$ with $F$ in \eqref{eq:ot-diagram}.

\begin{lemma}
The morphism \eqref{eq:ckl-pot} is isomorphic to the fixed part of $\phi_{F/B_W}$; that is, there is a commuting square
\[
\begin{tikzcd}
\EE_{B_Z/B_W}|_F^{\fix} \arrow[d, "\phi_{B_Z/B_W}|^{\fix}_F"] \arrow[r, "\sim"] & \EE_{F/B_W}^{\fix} \arrow[d, "\phi_{F/B_W}^{\fix}"]\\
\LL_{B_Z/B_W}|_F \arrow[r] & \LL_{F/B_W}
\end{tikzcd}
\]
\end{lemma}
\begin{proof}
Before applying the $\fix$ functor, this diagram is part of Lemma \ref{lem:functoriality1} with $Z,W$ (there) both equal to $Z$ (here), $V$ (there) equal to $W$ (here), $B_Z=F$, $B_W=B_Z$, and $B_V=B_W$ (continuing the same notational pattern). Note that in the notation of Lemma \ref{lem:functoriality1}, we have $\EE_{B_Z/B_W}=0$.
\end{proof}

\section{Absolute versus relative cosection localized virtual classes}
\label{sec:cosection-localization}
Let $Z$ be any smooth algebraic stack locally of finite type. Let $X\rightarrow Z$ be a morphism from a finite type, separated Deligne--Mumford stack $X$ with a relative perfect obstruction theory $\phi\colon \EE \rightarrow \LL_{X/Z}$. 
Let $\sigma_{X/Z}\colon \EE_{X/Z}^\vee \rightarrow \OO_X[-1]$ be a cosection (defined on all of $X$). Recall the functor $h^1/h^0(\bullet)$ from a certain subcategory of the derived category of $X$ to the category of abelian cone stacks on $X$ (see e.g. \cite[Prop~2.4]{BF97}). Applying this functor to $\sigma$ yields a map $\coneE_{X/Z}\rightarrow \CC_X$ of cone stacks on $X$, where we define $\coneE_{X/Z} =  h^1/h^0(\EE_{X/Z}) $ and $\CC_X = h^1/h^0(\OO_X[-1])$. We define the \textit{kernel} $\coneE_{X/Z}(\sigma)$ to be the fiber product
\[
\coneE_{X/Z}(\sigma) = \coneE_{X/Z}\times_{\CC_X, 0} X.
\]
(Note that the underlying set of $\coneE_{X/Z}(\sigma)$ is the locus $\mathbf{E}(\sigma)$ in \cite[(3.2)]{KiLi13}.) We do not require (as in \cite{KiLi13}) that $\sigma$ descends to an absolute cosection, but instead we directly assume that
\begin{equation}\label{eq:descends}
\text{the map of cone stacks}\;h^1/h^0(\sigma \circ \phi^\vee)\;\text{is zero.}
\end{equation}
In this case, by the universal property of $\coneE_{X/Z}(\sigma)$, the relative intrinsic normal sheaf is contained in $\coneE_{X/Z}(\sigma)$, and hence the relative intrinsic normal cone $\coneC_{X/Z}$ is as well. We define
\[
[X]^{\vir}_{\sigma_{X/Z}} := 0^!_{\sigma_{X/Z}}([\coneC_{X/Z}])\quad \text{in}\;A_*(X(\sigma))
\]
where $0^!_{\sigma_{X/Z}}$ is the localized Gysin map $s^!_{\coneE_{X/Z}, \sigma_{X/Z}}$ \cite[(3.5)]{KiLi13} and $X(\sigma)$ is the degeneracy locus.

The following proposition shows that in fact, $\sigma_{X/Z}$ defines a cosection $\sigma_X$ of an absolute obstruction theory $\EE_X$ induced by $\EE_{X/Z}$, that the absolute intrinsic normal cone $\coneC_X$ is contained in the absolute kernel $\coneE_X(\sigma)$, and finally that our definition of $[X]^\vir_\sigma$ agrees with that of \cite{KiLi13} (where $[X]^\vir_\sigma$ is defined to be $0^!_{\sigma_X}([\coneC_X])$. Our proposition is, however, more general than this and may be viewed as a cosection localized analog of \cite[Proposition~3]{KKP03}.

\begin{proposition}\label{prop:cos-loc-func}
  Fix $Z$, a smooth algebraic stack locally of finite type.
 Let $X$ be a Deligne--Mumford stack and let $X \rightarrow Y \rightarrow Z$ be morphisms such that $X \rightarrow Y$ and $X \rightarrow Z$ are separated and finite type, and $Y\rightarrow Z$ is smooth and locally finite type.
 Let $\phi\colon \EE_{X/Y} \rightarrow \LL_{X/Y}$ be a relative perfect obstruction theory and $\sigma\colon \EE_{X/Y}^\vee \rightarrow \OO[-1]$ a cosection such that $h^1/h^0(\sigma \circ \phi^\vee)=0$. Then there is a relative perfect obstruction theory $\psi\colon \FF_{X/Z} \rightarrow \LL_{X/Z}$ with induced cosection $\rho$ satisfying $h^1/h^0(\rho \circ \psi^\vee)=0$. Moreover, the degeneracy loci $X(\sigma)$ and $X(\rho)$ are equal, and the cosection localized virtual cycles induced by $(\phi, \sigma)$ and $(\psi, \rho)$ agree.
\end{proposition}

\begin{proof}
Let $q$ be the map from $X \rightarrow Y$. The relative obstruction theory $\psi\colon \FF_{X/Z} \rightarrow \LL_{X/Z}$ is any morphism fitting into a morphism of distinguished triangles
\begin{equation}\label{eq:cos-loc-func1}
\begin{tikzcd}
\FF_{X/Z} \arrow[r] \arrow[d, "\psi"] & \EE_{X/Y} \arrow[r, "\delta"] \arrow[d, "\phi"] & q^*\LL_{Y/Z}[1] \arrow[r] \arrow[d] & {}\\
\LL_{X/Z} \arrow[r] & \LL_{X/Y} \arrow[r] & q^*\LL_{Y/Z}[1] \arrow[r] &{}
\end{tikzcd}
\end{equation}
By the proof of \cite[Proposition~3]{KKP03}, $\psi$ is indeed a perfect obstruction theory. Dualizing, we see that $\sigma \circ \delta^\vee$ factors through $\sigma \circ\phi^\vee$ and hence by assumption $h^1/h^0(\sigma \circ \delta^\vee)=0$. By \cite[Lem~2.2]{BF97}, the map $\sigma \circ \delta^\vee$ is nullhomotopic, and there is a (not necessarily unique) morphism $\rho$ inducing a morphism of distinguished triangles as below.
\begin{equation}\label{eq:cos-loc-func2}
\begin{tikzcd}
\OO_X[-1] & \OO_X[-1] \arrow[l] & 0 \arrow[l] & \arrow[l] {}\\
\FF_{X/Z}^\vee \arrow[u, dashrightarrow, "\rho"] & \EE_{X/Y}^\vee \arrow[l] \arrow[u, "\sigma"] & (q^*\LL_{Y/Z}[1])^\vee \arrow[u] \arrow[l, "\delta^\vee"'] & {}\arrow[l]
\end{tikzcd}
\end{equation}
By \cite[Proposition~2.7]{BF97}, from \eqref{eq:cos-loc-func1} and \eqref{eq:cos-loc-func2} we get a commuting diagram of abelian cone stacks whose rows are short exact sequences and $\mathfrak{K} = h^1/h^0((q^*\LL_{Y/Z}[1])^\vee)$:
\begin{equation}\label{eq:cos-loc-func3}
\begin{tikzcd}
0 & \arrow[l] \CC_X & \arrow[l, "\sim"] \CC_X & 0 \arrow[l] & \arrow[l] 0\\
0 & \arrow[l] \arrow[u, "\rho"]h^1/h^0(\FF_{X/Z}^\vee)& \arrow[l] \arrow[u, "\sigma"]h^1/h^0(\EE_{X/Y}^\vee) & \arrow[u] \mathfrak{K}\arrow[l] & \arrow[l] 0\\
0 & \arrow[l]\mathfrak{N}_{X/Z} \arrow[u, "\psi^\vee"]& \arrow[l,"\alpha"] \mathfrak{N}_{X/Y} \arrow[u, "\phi^\vee"]&  \arrow[l] \mathfrak{K}\arrow[u, equal]& \arrow[l] 0\\
\end{tikzcd}
\end{equation}
Here $\mathfrak{N}$ denotes an intrinsic normal sheaf. In particular the bottom left square is fibered and the horizontal arrows in that square are smooth surjections. Because the composition $\sigma \circ \phi^\vee$ is zero and $\alpha$ is surjective, we also have $\rho \circ \psi^\vee =0.$

On the one hand, by the proof of \cite[Proposition~3]{KKP03} we may replace the bottom row of \eqref{eq:cos-loc-func3} with the exact sequence $0 \leftarrow \mathfrak{C}_{X/Z} \leftarrow \mathfrak{C}_{X/Y} \leftarrow \mathfrak{K}\leftarrow 0$. On the other hand, we have a diagram where all squares are fibered and the arrow $h^1/h^0(\EE^\vee_{X/Y}) \rightarrow h^1/h^0(\FF^\vee_{X/Z})$ is a smooth surjection:
\[
\begin{tikzcd}
\mathfrak{K} \arrow[r] \arrow[d] & h^1/h^0(\EE_{X/Y}^\vee)(\sigma) \arrow[r] \arrow[d] & h^1/h^0(\EE_{X/Y}^\vee) \arrow[d] \\
X \arrow[r, "0"] &h^1/h^0(\FF^\vee_{X/Z})(\rho) \arrow[r] \arrow[d] & h^1/h^0(\FF_{X/Z}^\vee) \arrow[d, "\rho"] \\
&X \arrow[r, "0"] & \CC_X
\end{tikzcd}
\]
To identify the two leftmost terms in the top row, observe that the composition of the two right vertical arrows is $\sigma$ and the composition of the two middle horizontal arrows is the zero section.
Using \cite[Def~1.12]{BF97}, we obtain a short exact sequence 
\begin{equation}\label{eq:cos-loc-func4}0 \leftarrow h^1/h^0(\FF_{X/Z}^\vee)(\rho) \leftarrow h^1/h^0(\EE_{X/Y}^\vee)(\sigma) \leftarrow \mathfrak{K}\leftarrow 0.
\end{equation}
Finally, because $\rho\circ \psi^\vee = \sigma \circ \phi^\vee=0$, 
we may replace the middle row of \eqref{eq:cos-loc-func3} with \eqref{eq:cos-loc-func4}. We get a fiber square with horizontal arrows that are smooth surjections:
\[
\begin{tikzcd}
h^1/h^0(\EE_{X/Y}^\vee)(\rho) \arrow[r, "\alpha"] & h^1/h^0(\FF_{X/Z}^\vee)(\sigma) \\
\mathfrak{C}_{X/Y} \arrow[r] \arrow[u] & \mathfrak{C}_{X/Z} \arrow[u]
\end{tikzcd}
\]
From here, we can compute the cosection localized virtual class:
\[
[X]^{\vir}_{\rho} = 0^!_{\rho}[\mathfrak{C}_{X/Y}] = 0^!_{\rho}\alpha^*[\mathfrak{C}_{X/Z}] = 0^!_{\sigma}[\mathfrak{C}_{X/Z}] = [X]^{\vir}_{\sigma}.
\]
Here, $0^!_{\sigma}$ and $0^!_{\rho}$ are the cosection localized Gysin maps of \cite[Section~2]{KiLi13}. The second equality above is the compatibility of these maps with the usual Gysin maps.
\end{proof}

\bibliographystyle{abbrv}
\bibliography{cosec}

\end{document}